\documentclass[11pt,reqno]{amsart}
\usepackage{amsmath}
\usepackage{amsfonts}
\usepackage{amssymb}
\usepackage[dvips]{graphicx}
\usepackage{color}
\usepackage[pagewise]{lineno}
\usepackage{url}
\theoremstyle{plain}
\newtheorem{athm}{Theorem}

\newtheorem{theorem}{Theorem}[section]
\newtheorem{proposition}[theorem]{Proposition}
\newtheorem{corollary}[theorem]{Corollary}
\newtheorem{lemma}[theorem]{Lemma}

\theoremstyle{definition}

\newtheorem{example}{Example}[section]
\newtheorem{definition}[theorem]{Definition}

\newtheorem{remark}[theorem]{Remark}

\DeclareMathOperator{\esssup}{ess \ sup}

\DeclareMathOperator{\diam}{diam}

\DeclareMathOperator{\ho}{H}
\input{tcilatex}

\begin{document}

\begin{abstract}
	We consider a class of endomorphisms which contains a set of piecewise partially hyperbolic skew-products with a  non-uniformly expanding base map. The aimed transformation preserves a foliation which is almost everywhere uniformly contracted with possible discontinuity sets, which are parallel to the contracting direction. We
prove that the associated transfer operator, acting on suitable anisotropic normed spaces,
has a spectral gap (on which we have quantitative estimation) and the disintegration of the unique invariant physical measure, along the stable leaves, is $\zeta$-H\"older. We use this fact to obtain exponential decay of correlations on the set of $\zeta$-H\"older functions.   
\end{abstract}

\title[Decay of Correlations and H\"older Regularity for Skew Products]{H\"older Regularity and Exponential Decay of Correlations for a class of Piecewise Partially Hyperbolic Maps.}
\author[Rafael A. Bilbao]{Rafael A. Bilbao}
\author[Ricardo Bioni]{Ricardo Bioni}
\author[Rafael Lucena]{Rafael Lucena}

\date{\today }

\address[Rafael A. Bilbao]{Universidad Pedag\'ogica y Tecnol\'ogica de Colombia, Avenida Central del Norte 39-115, Sede Central Tunja, Boyac\'a, 150003, Colombia.}
\email{rafael.alvarez@uptc.edu.co}

\address[Ricardo Bioni]{Rua Costa Bastos, 34, Santa Teresa, Rio de Janeiro-Brasil}
\email{ricardo.bioni@hotmail.com}

\address[Rafael Lucena]{Universidade Federal de Alagoas - UFAL, Av. Lourival Melo Mota, S/N
	Tabuleiro do Martins, Maceio - AL, 57072-900, Brasil}
\email{rafael.lucena@im.ufal.br}
\urladdr{www.im.ufal.br/professor/rafaellucena}
\maketitle


\section{Introduction}

In this paper, we take advantage of the new ideas on the construction of anisotropic spaces, introduced by \cite{GLu}, to study the behaviour of the transfer operator associated to maps $F:\Sigma \longrightarrow \Sigma$ which has partially hyperbolic skew product structure with a non-uniformly expanding quotient map $f: M \longrightarrow M$. In other words, we prove spectral gap for the transfer operator associated to $F$ and explore some consequences. For instance, we prove that this sort of system has a physical measure which admits a H\"older regular disintegration along the stable foliation. Here, this is enough to have some other statistical properties like exponential decay of correlations on the set of $\zeta$-H\"older functions.

Usually, this sort of study is carried out as an application of the Ionescu-Tulcea and Marinescu's Theorem, by constructing a pair of suitable spaces of functions, a stronger and an auxiliary weaker space, such that the action of the Perron-Frobenius operator on the stronger space has spectral gap (see \cite{Ba}, \cite{L2}, \cite{BG}, \cite{G} and \cite{V} for some introductory texts).

In recent years, the technique of finding good anisotropic norms was extended to piecewise hyperbolic systems (see e.g. \cite%
{BT}, \cite{BaG}, \cite{BaG2}, \cite{DL}, \cite{GL} and \cite{B},\cite{D} for recent papers containing a survey of the topic). From these properties, several limit theorems or stability statements can be deduced. In these approaches, the existence of an expanding direction, enough regularity of the system or transversality between the map's singular set and the contracting directions is required, which make them not suitable for our situation.

Despite not giving explicit bounds on the rate of decay of correlations, the application of the ITM's Theorem has shown to be fruitful, especially for the study of piecewise expanding maps (see \cite{RE}, for instance). Alternatively, \cite{LV2} took advantage of Garrett Birkhoff's idea, and by defining a Hilbert metric on a cone of functions restricted to which the \textbf{PF} operator maps the cone strictly inside itself, explicit bounds on the rate of decay of correlations was obtained for a class of non-Markov piecewise expanding systems on the interval.

The Hilbert metric's idea has shown to be powerful and became quite standard. Moreover, its application bypasses the boundary of the expanding maps. An important application of this technique was given in \cite{VM} and \cite{VAC} in the analysis of the spectrum of the Ruelle-Perron-Frobenius operator, in order to study the Thermodynamic Formalism for a sort of mostly expanding maps. They prove, among other results, statements on exponential decay of correlations, central limit theorems and stability. Another application of the projective metric's technique was given in \cite{TC}. That work deals with diffeomorphisms onto its image and proves exponential decay of correlations for H\"older
continuous observables and central limit theorem for the maximal entropy
probability measure. 

In this sense, since the maps presented here are not too regular in horizontal direction and maybe not invertible, the results reached in the present work generalize the ones given in \cite{TC}. Here, we deal with endomorphisms $F$ with discontinuity sets (if non-empty) parallel to the stable direction. We study the action of the transfer operator of $F:M\times K \longrightarrow M \times K$ on a suitable normed space of signed measures. These vector spaces are constructed by identifying a measure on the square, $M\times K$, with a path of measures $M\longmapsto SM(K)$ ($SM(K)$ denotes the space of signed measures on $K$), where $SM(K)$ is endowed with the ``dual of H\"older" norm.  This is achieved by generalizing the Wasserstein-Kantorovich-like norm defined in \cite{GLu} (see also \cite{ben} for similar applications of the Wasserstein distance to obtain limit theorems) and this is enough to obtain the limit theorems on a larger Banach space of functions as an application of Theorem \ref{regg}, where we prove that the disintegration of the physical $F$-invariant measure is $\zeta$-H\"older regular. This sort of result has many other applications. For example, the reader can see \cite{GLu} and \cite{Gjep}, where a stability result was proved for Lorenz-like systems and partially hyperbolic skew products, respectively, under \textit{ad-hoc} perturbations of the system.

\noindent\textbf{Statements of the Main Results.} Here we expose the main results of this work. We note that, the hypothesis (f1), (f2), (f3) and (G1) on the system $F$ will be stated in section \ref{sec1} and (G2) in section \ref{hhhhhhhh}.

The first result guaranties existence and uniqueness for the $F$-invariant measure in the space $S^{\infty}$ (see equation (\ref{sinfi})). Its proof is given in section \ref{invt}.    

\begin{athm}
	\label{belongss} There exists a unique $F$-invariant probability, $\mu_0 \in S^{\infty }$. Moreover, $\pi_{x*}\mu_0 = m_1$.
\end{athm}

Next result, where the proof can be found in section \ref{ierutierutiuyerit}, shows that the transfer operator acting on the space $S^{\infty}$ is quasi-compact. This sort of result has many consequences for the dynamic and it implies several limit theorems. For instance, we obtain an exponential rate of convergence for the limit $$\lim_{n \to \infty} {C_n(f,g)}=0,$$ where $$C_n(f,g):=\left| \int{(g \circ F^n )  f}d\mu_0 - \int{g  }d\mu_0 \int{f  }d\mu_0 \right|,$$ $g: M \times K \longrightarrow \mathbb{R} $ is a $\zeta$-Holder function and $f \in \Theta _{\mu _0}$. The set $\Theta _{\mu _0}$ is defined as 

\begin{equation*}
\Theta _{\mu _0}:= \{ f: M \times K \longrightarrow \mathbb{R}; f\mu_0 \in S^\infty\},
\end{equation*}where the measure $f\mu_0$ is defined by $f\mu_0(E):=\int _E{f}d\mu_0$, for all measurable set $E$.

\begin{athm}[Spectral gap on $S^{\infty}$]
	\label{spgapp}If $F:\Sigma \longrightarrow \Sigma$ satisfies (f1), (f2), (f3) and (G1)
	given at beginning of section \ref{sec2}, then the operator $\func{F}_{\ast
	}:S^{\infty}\longrightarrow S^{\infty}$ can be written as 
	\begin{equation*}
	\func{F}_{\ast }=\func{P}+\func{N},
	\end{equation*}%
	where
	
		\begin{enumerate}
		\item[a)] $\func{P}$ is a projection, i.e., $\func{P} ^2 = \func{P}$ and $\dim
		\func{P}(S^{\infty})=1$;
		
		\item[b)] there are $0<\xi <1$ and $K>0$ such that $\forall \mu \in S^\infty$ 
		\begin{equation*}
		||\func{N}^{n}(\mu )||_{S^{\infty}}\leq ||\mu||_{S^{\infty}} \xi ^{n}K;
		\end{equation*}
		
		\item[c)] $\func{P}\func{N}=\func{N}\func{P}=0$.
	\end{enumerate}
\end{athm}

Next proposition is a consequence of all previous theorems. It shows that the system $F$ has exponential decay of correlations ($\displaystyle{\lim _{n \to \infty} {C_n(f,g)}=0}$ exponentially fast) for observables $f \in \Theta _{\mu _0}$ and $g\in \ho_\zeta(\Sigma)$. Its proof is given in section \ref{decay}.

\begin{athm}
	\label{decay1}
	For every $\zeta$-H\"older function $g: \Sigma \longrightarrow \mathbb{R} $ and all $f \in \Theta _{\mu _0}$, it holds $$\left| \int{(g \circ F^n )  f}d\mu_0 - \int{g  }d\mu_0 \int{f  }d\mu_0 \right| \leq ||f \mu _0||_{S^{\infty}} K |g|_{\zeta}  \xi ^{n} \ \ \forall n \geq 1,$$where $\xi$ and $K$ are from Theorem \ref{spgapp} and $|g|_{\zeta} := |g|_\infty + H_\zeta(g)$. 
\end{athm}

The following theorem is an estimate for the Holder's constant (see equation (\ref{Lips2}) in Definition \ref{Lips3}) of the disintegration of the unique $F$-invariant measure $\mu_0 $ in $S^{\infty}$. This kind of result has many applications and similar estimations (for other systems) were given in \cite{GLu} and \cite{BM}. In \cite{GLu}, for instance, they use the regularity of the disintegration to prove stability of the $F$-invariant measure under a kind of \textit{ad-hoc} perturbation. Here, we use this result to show that the abstract set $\Theta _{\mu _0}$, defined above, contains the $\zeta$-Holders functions. The proof of the following result is presented in section \ref{hhhhhhhh}.

\begin{athm}
	Suppose that $F:\Sigma \longrightarrow \Sigma$ satisfies (f1), (f2), (f3), (G1), (G2) and $(\alpha \cdot L)^\zeta<1$ and consider the unique $F$-invariant probability $\mu _{0}\in S^\infty$. Then $\mu _{0}\in \mathcal{H} _\zeta^{+}$ and 
	\begin{equation*}
	|\mu _{0}|\leq \dfrac{D}{1-\beta},
	\end{equation*}where $D$ and $\beta$ are from Proposition \ref{iuaswdas}.
	\label{regg}
\end{athm}

As a consequence of the estimative given in the previous theorem, next theorem complements the Theorem \ref{decay1}. It implies, for instance, $$\lim _{n \to \infty} {C_n(f,g)}=0$$exponentially fast for all Holder observable $f, g \in \ho_\zeta(\Sigma)$. Its demonstration is given in section \ref{last123}.

\begin{athm}\label{disisisii}
	Suppose that $F:\Sigma \longrightarrow \Sigma$ satisfies (f1), (f2), (f3), (G1), (G2) and $(\alpha \cdot L)^\zeta<1$ and let $\mu_0$ be the unique $F$-invariant measure in $S^\infty$. Then, $\ho_\zeta(\Sigma) \subset \Theta _{\mu_0}$.
\end{athm}

\noindent\textbf{Plan of the paper.} The paper is structured as follows:

\begin{itemize}
	\item Section 2: we introduce the kind of systems we consider in the
	paper. Essentially, it is a class of systems which contains a set of piecewise partially hyperbolic dynamics ($F(x,y)=(f(x), G(x,y))$) with a non-uniformly expanding basis map, $f$, and whose fibers are uniformly contracted $m_1$-a.e, where $m_1$ is an $f$-invariant probability measure. Here and until section 7, where more regularity is required, we do not ask for any kind of regularity on $G$ in the horizontal direction (for the functions $x \longmapsto G(x, y)$, $y$ fixed);
	
	\item Section 3: we introduce the functional spaces used in the paper and
	discussed in the previous paragraphs;
	
	\item Section 4: we show the basic properties of the transfer operator of $F$
	when applied to these spaces. In particular we see that there is a useful
	\textquotedblleft Perron-Frobenius\textquotedblright-like formula (see
	Proposition \ref{niceformulaab});
	
	\item Section 5: we discuss the basic properties of the iteration of the
	transfer operator on the spaces we consider. In particular, we prove a \emph{%
		Lasota-Yorke inequality and a convergence to equilibrium statement} (see
	Propositions \ref{lasotaoscilation2} and \ref{5.8});
	
	\item Section 6: we use the convergence to equilibrium and the
	Lasota-Yorke inequalities to prove the \emph{spectral gap} for the transfer
	operator associated to the system restricted to a suitable strong space (see
	Theorem \ref{spgapp}) and prove a decay of correlation statement on an abstract set of functions;
	
	\item Section 7: we consider a similar system with some more
	regularity on the family of functions $\{G(\cdot,y)\}_{y \in K}$, $x \longmapsto G(x, y)$: there exists a partition (into open sets) $\mathcal{P} = P_1, \cdots, P_{\deg (f)}$, such that the restriction of the function $x \longmapsto G(x, y)$ to $P_i$ is $(k_{y,i}, \zeta)$-H\"older, where the family $\{k_{y,i}\}_{y\in K}$ is bounded, and it holds for all $i$. This allows discontinuities on the boundaries $(\partial P_i) \times K$. For this sort of system, we prove a stronger regularity result for the iteration of probability measures (see Corollary \ref{kjdfhkkhfdjfh} and Remark \ref{kjedhkfjhksjdf}) and show that the $F$-invariant physical measure has a $\zeta$-H\"older disintegration along the stable fibers (see Theorem \ref{regg});
	
	\item Section 8: we use the $\zeta$-H\"older regularity of the physical measure established in section 7, to prove that the abstract set of functions on which the system has decay of correlations contains all $\zeta$-H\"older functions (see Theorem \ref{disisisii}).
\end{itemize}

\textbf{Acknowledgment} We are thankful to Stefano Galatolo for all valuable comments and fruitful discussions regarding this work.

This work was partially supported by Alagoas Research Foundation-FAPEAL (Brazil) Grants 60030 000587/2016, CNPq (Brazil) Grants 300398/2016-6, CAPES (Brazil) Grants 99999.014021/2013-07 and EU Marie-Curie IRSES Brazilian-European partnership in
Dynamical Systems (FP7-PEOPLE- 2012-IRSES 318999 BREUDS).

\section{Settings\label{sec1}}

Fix a compact and connected Riemannian manifold, $M$, equipped with its Riemannian metric $d_1$. For the sake of simplicity, we suppose that $\diam(M) = 1$, this is not restrictive but will avoid multiplicative constants. Moreover, consider a compact metric space $(K,d_2)$, endowed with its Borel's sigma algebra, $\mathcal{B}$. We set $\Sigma := M\times K$ and we endow this space with the metric $d_1 + d_2$.

\subsection{Contracting Fiber Maps with Non Uniformily Expanding Basis\label{sec2}}

Let $F$ be the map $F:\Sigma \longrightarrow \Sigma$ given by
\begin{equation}\label{cccccc}
F(x,z)=(f(x),G(x,z)),
\end{equation}where $G: \Sigma \longrightarrow K$ and $f:M\longrightarrow M$ are measurable maps satisfying what follows.

\subsubsection{Hypothesis on $f$}\label{hf}
Suppose that $f:M \longrightarrow M$ is a local diffeomorphism and assume that there is a continuous function $L:M\longrightarrow\mathbb{R}$, s.t. for every $x \in M$ there exists a neighbourhood $U_x$, of $x$, so that $f_x:=f|_{U_x}: U_x \longrightarrow f(U_x)$ is invertible and $$d_1(f_x^{-1}(y), f_x^{-1}(z)) \leq L(x)d_1(y, z), \ \ \forall y,z \in f(U_x).$$

In particular, $\#f^{-1}(x)$ is constant for all $x \in M$. We set $\deg(f):=\#f^{-1}(x)$, the degree of $f$.

Denote by 
\begin{equation}\label{ro}
\rho(\gamma) := \dfrac{1}{|\det (f^{'}(\gamma))|},
\end{equation}where $\det (f^{'})$ is the Jacobian of $f$ with respect to $m_1$.

Suppose that there is an open region $\mathcal{A} \subset M$ and constants $\sigma >1$ and $L\geq 1$ such that  
\begin{enumerate}
	\item[(f1)] $L(x) \leq L$ for every $x \in \mathcal{A}$ and $L(x) < \sigma ^{-1}$ for every $x \in \mathcal{A}^c$. Moreover, $L$ is close enough to $1$: the precise estimation for $L$ is given in equation (\ref{kdljfhkdjfkasd});
	\item[(f2)] There exists a finite covering $\mathcal{U}$ of $M$, by open domains of injectivity for $f$, such that $\mathcal{A}$ can be covered by $q<\deg(f)$ of these domains.
\end{enumerate}

Denote by $H_\zeta$ the set of the H\"older functions $h:M \longrightarrow \mathbb{R}$, i.e., if we define $$H_\zeta(h) := \sup _{x\neq y} \dfrac{|h(x) - h(y)|}{d_1(x,y)^\zeta},$$then

$$H_\zeta:= \{ h:M \longrightarrow \mathbb{R}: H_\zeta(h) < \infty\}.$$

Next, (f3) is an open condition relatively to the H\"older norm and equation (\ref{f32}) means that $\rho$ belongs to a small cone of H\"older continuous functions (see \cite{VAC}).

\begin{enumerate}
	
	\item[(f3)] There exists a sufficiently small $\epsilon _\rho >0$ s.t. 
	
	\begin{equation}\label{f31}
	\sup \log (\rho) - \inf \log (\rho) <\epsilon _\rho;
	\end{equation}and
	\begin{equation}\label{f32}
	H_\zeta (\rho) < \epsilon_\rho \inf \rho.
	\end{equation}
\end{enumerate}

Precisely, we suppose the constants $\epsilon _\rho$ and $L$ satisfy the condition 
\begin{equation}\label{kdljfhkdjfkasd}
\exp{\epsilon_\rho} \cdot \left( \dfrac{(\deg(f) - q)\sigma ^{-\alpha} + qL^\alpha[1+(L-1)^\alpha] }{\deg(f)}\right)< 1.
\end{equation}

According to \cite{VAC}, such a map (satisfying (f1), (f2) and (f3)) $f:M \longrightarrow M$ has an invariant probability $m_1$ of maximal entropy, absolutely continuous with respect to a conformal measure, and its Perron-Frobenius operator with respect to $m_1$, $\func{P}_f: L^1_{m_1} \longrightarrow L^1_{m_1}$, defined for $\varphi  \in L^1_{m_1}$ by $\func{P}_f (\varphi)(x) = \sum _{i=1}^{\deg(f)}{\varphi (x_i)\rho (x_i)}$ ($x_i$ is the $i$-th pre image of $x$, $i=1, \cdots, \deg(f)$), satisfies the following result.

\begin{theorem}\label{loiub}
	There exist $0< r<1$ and $D>0$ s.t. for all $\varphi \in H_\zeta$, with $\int{\varphi}dm_1 =0$, it holds $$|\func{P_f}^n(\varphi)|_\zeta \leq Dr^n|\varphi|_\zeta \ \ \forall \ n \geq 1,$$where $|\varphi|_\zeta := H_\zeta (\varphi) + |\varphi|_{\infty} $ for all $\varphi \in H_\zeta$.  
\end{theorem}

\begin{remark}\label{chkjg} 
	
	By (f2),  (see Lemma \ref{jdjkfklfd}) there exists a disjoint finite family, $\mathcal{P}$, of open sets, $P_1, \cdots, P_{\deg{(f)}}$, s.t. $\bigcup_{i=1}^{\deg{(f)}} P_i=M$ $m_1$-a.e., and $f|_{P_{i}}:P_i \longrightarrow f(P_i)$ is a diffeomorfism for all $i=1, \cdots \deg{(f)}$. Moreover, $f(P_i)=M$ $m_1$-a.e., for all $i=1, \cdots, \deg(f)$. Therefore, it holds that $$\func{P}_f(\varphi)(x) = \sum _{i=1}^{\deg{(f)}} {\varphi (x_i)\rho (x_i)}\chi_{f(P_{i})}(x),$$ for $m_1$-a.e. $x \in M$, where $$\rho_i(\gamma):= \frac{1}{|\det (f_i^{'}(\gamma))|}$$ and $f_i = f|_{P_i}$. This expression will be used later on.
\end{remark}

\begin{lemma}\label{jdjkfklfd}
	Let $(f,M,m_1)$ be a non-singular system ($m_1 (A)=0 \Rightarrow m_1 (f^{-1}(A))=0$, where $M$ is a compact and connected manifold and suppose that $m_1$ be absolutely continuous with respect to Lebesgue. Suppose that $f:M\to M$ a local homeomorphism with degree $\deg(f)$. Then, there exist a finite partition $\mathcal{P}=\{P_1, \cdots, P_p\}$, where $p = \deg(f)$, $P_i$ is open for all $i= 1, \cdots, \deg(f)$ and $m_1(M \setminus f(P_i))=0$ for all $i= 1, \cdots, \deg(f)$.  
\end{lemma}

\begin{proof}
	Since $f$ is a local homeomorphism, $a\in M$, there exists $r_a > 0$ s.t.$f:V(a^j)\to B(a,r_a)$ is a homeomorphism for all $j=1,...,p$, where $f^{-1}(a)=\{a^1 , ..., a^p  \}$ and $V(a^j)\cap V(a^i)=\emptyset$, for all $j\ne i$. On the other hand, $M=\bigcup_{a\in M}B(a, r_a)$ and by compacity there exist a finite number of balls $B(a,r_a)$ which covers $M$, i.e., $M=\bigcup_{n=1}^{r}B(a_n , r_{a_n})$, where $a_n \in M$ for all $n=1,\cdots,r$. From this cover, let $\mathcal{K}$ be a finite partition such that $\#\mathcal{K}=:q$ and the boundary of each atom has $m_1$-null measure.

	Therefore, for each $K_t\in \mathcal{K}$ we have $f^{-1}(K_t)= K_t^1 \cup ...\cup K_t^p$ where $K_t^j \cap K_t^i =\emptyset$ for all $j\ne i$ and $1\leq t \leq q$. Thus, $f:K_{t}^j \to K_t$ is a homeomorphism for each $1\leq t \leq q$ and $1\leq j \leq p$. Define $M'= \bigcup_{K\in \mathcal{K}} \func{int}(K)$ in a way that $m_1 (M\setminus M')=0$, where $M\setminus M'$ is the boundary of the sets $K\in \mathcal{K}$. Then, define
	\[
	P^1 =K_{1}^1 \cup  K_{2}^1 \cup K_{3}^1 \cup \dots \cup K_{q}^1,  
	\]
	
	\[
	P^2 =K_{1}^2 \cup  K_{2}^2 \cup K_{3}^2 \cup \dots \cup K_{q}^2,  
	\]
	\begin{center}
		. \\
		. \\
		. \\
	\end{center}

	\[
	P^p =K_{1}^p \cup  K_{2}^p \cup K_{3}^p \cup \dots \cup K_{q}^p.  
	\]
	
	Finally, we have that $\mathcal{P}:=\{ P^j\}_{1\leq j \leq p}$ is a measurable partition of $M$ and by construction  $f:\func{int}(P^j) \to M'$ is a homeomorphism for all $1\leq j \leq p$.
\end{proof}

\subsubsection{Hypothesis on $G$}
We suppose that $G: \Sigma \longrightarrow K$ satisfies:

\begin{enumerate}
	\item [(G1)] $G$ is uniformly contracting on $m_1$-a.e. vertical fiber, $\gamma_x :=\{x\}\times K$. 
\end{enumerate}Precisely, there is $0< \alpha <1$ such that for $m_1$-a.e. $x\in M$ it holds%
\begin{equation}
d_2(G(x,z_{1}),G(x,z_2))\leq \alpha d_2(z_{1},z_{2}), \quad \forall
z_{1},z_{2}\in K.  \label{contracting1}
\end{equation}We denote the set of all vertical fibers $\gamma_x$, by $\mathcal{F}^s$: $$\mathcal{F}^s:= \{\gamma _x:=\{ x\}\times K; x \in M \} .$$ When no confusion is possible, the elements of $\mathcal{F}^s$ will be denoted simply by $\gamma$, instead of $\gamma _x$.

\begin{example}\label{sesprowerpo}
	Let $f_0: \mathbb{T}^d \longrightarrow \mathbb{T}^d$ be a linear expanding map. Fix a covering $\mathcal{P}$ and an atom $P_1 \in \mathcal{P}$ which contains a periodic point (maybe fixed point) $p$. Then, consider a perturbation $f$, of $f_0$, inside $P_1$ by a pitchfork bifurcation, in a way that $p$ becomes a saddle for $f$. Therefore, $f$ coincides with $f_0$ in $P_1^c$, where we have uniform expansion. The perturbation can be made in a way that (f1) is satisfied, i.e., is never too contracting in $P_1$ and $f$ is still a topological mixing. Note that a small perturbation with the previous properties may not exist. If it does, then (f3) is satisfied. In this case, $m_1$ is absolutely continuous with respect to the Lebesgue measure which is an expanding conformal and positive measure on open sets. Hence, there can be no periodic attractors.
\end{example}
\begin{example}\label{sesprowerpoo}
In the previous example, assume that $f_0$ is diagonalisable, with eigenvalues $1< 1+ a<\lambda$, associated to $e_1, e_2$ respectively, and $x_0$ is a fixed point. Fix $a,\epsilon> 0$ such that $\log(\frac{1+ a}{1- a})<\epsilon$ and
\begin{equation*}
	\exp\epsilon\left(\frac{(\deg(f_0)- 1)(1+ a)^{-\alpha}+ (1/(1- a))^{\alpha}[1+ (a/(1- a))^\alpha]}{\deg(f_0)}\right)< 1.
\end{equation*}

Note that any smaller $a> 0$ will still satisfy these equations.

Let $\mathcal U$ be a finite covering of $M$ by open domains of injectivity for $f$. Redefining sets in $\mathcal U$, we may assume $x_0= (m_0, n_0)$ belongs to exactly one such domain $U$. Let $r> 0$ be small enough that $B_{2r}(x_0)\subset U$. Define $\rho=\eta_r\ast g$, where $\eta_r(z)= (1/r^2)\eta(z/r)$, $\eta$ the standard mollifier, and
\begin{equation*}
	g(m,n)=\begin{cases}
		\lambda(1- a),&\text{if }(m, n)\in B_r(x_0);\\
		\lambda(1+ a),&\text{otherwise.}
	\end{cases}
\end{equation*}
Finally, define a perturbation $f$ of $f_0$ by
\begin{equation*}
	f(m, n)= (m_0+\lambda (m- m_0),n_ 0+(\rho(m, n)/\lambda) (n- n_0)).
\end{equation*}
Then $x_0$ is a saddle point of $f$ and the desired conditions are satisfied for $\mathcal A= B_{2r}(x_0)$, $L= 1/(1- a)$ and $\sigma= 1+ 2a$. The only non-trivial condition is (f3). To show it, note that


\begin{equation*}
	\rho(x)-\rho(y)=\int_S \frac{2a}{\lambda(1- a^2)}\eta_r(z)\,dz-\int_{S'}\frac{2a}{\lambda(1- a^2)}\eta_r(z)\,dz,
\end{equation*}
where $S=\{z\in\mathbb R^2: x- z\in B_r(x_ 0), y- z\notin B_r(x_ 0)\}$ and $S'=\{z\in\mathbb R^2: y- z\in B_r(x_ 0), x- z\notin B_r(x_ 0)\}$. Take $x, y\in\mathbb R^2$ and write $|x- y|= qr$, $A_q=\{z\in\mathbb R^2: 1- q< |z|< 1\}$. We have
\begin{equation*}
	\frac{|\rho(x)-\rho(y)|}{|x- y|^\zeta}\leq\frac{2a\eta_r(S)}{\lambda(1- a^2)q^\zeta r^\zeta}\leq\frac{2a\eta(A_q)/q^\zeta}{\lambda(1- a^2)}.
\end{equation*}
Since $N=\sup_{q> 0}\eta(A_q)/q^\zeta< +\infty$, we can take $a$ so small that $2aN/(1- a)<\epsilon$, therefore $H_\zeta(\rho)<\epsilon\inf\rho$.
\end{example}

\subsection{A Lasota-Yorke Inequality for $\func{P}_f$}

In this section, we provide a Lasota-Yorke inequality for the Perron-Frobenius Operator of $f$, acting on the space of H\"older functions. The proof follows from a general approach, where such a linear operator has spectral gap.

\begin{lemma}\label{contracaoimplicaly}
	Let $(B_w,\|\cdot\|_w)$ and $(B_s,\|\cdot\|_s)$ be two Banach spaces, with $B_s\subset B_w$ and let $L: B_w\to B_w$ be a bounded linear operator s.t. $L(B_s)\subset B_s$. Suppose there exists a projection $P: B_w\to B_w^1\subset B_s$, bounded with respect to the norms $||\cdot ||_w$ and $||\cdot ||_{B_w \rightarrow B_s}$, such that $L(B_w^1)\subset B_w^1$.
	
	If there exist $q< 1$ and $Q\in\mathbb{R}$ such that
	\begin{equation}\label{contracao}
	\forall n\in \mathbb{N}: \|L^n|_{B_s\cap B_w^0}\|_s\leq Qq^n,
	\end{equation}
	where $B_w^0:= P^{-1}(0)$, then there exists $N\in\mathbb{N}$, $r< 1$ and $R\in\mathbb{R}$ s.t.
	\begin{equation}\label{lasota-yorke}
	\|L^N g\|_s\leq r\|g\|_s + R\|g\|_w
	\end{equation} 
	for all $g\in B_s$.
\end{lemma}
\begin{proof}
	For each $g\in B_s$ and $n\in\mathbb{N}$ we have that
	\begin{equation*}
	\|L^n(g- Pg)\|_s\leq Qq^n\|g- Pg\|_s.
	\end{equation*}
	by the triangle inequality,
	\begin{gather*}
	\|g- Pg\|_s\leq\|g\|_s+\|Pg\|_s;\\
	\|L^n g\|_s\leq\|L^n(g- Pg)\|_s+\|L^n Pg\|_s.
	\end{gather*}
	Since $L(B_w^1)\subset B_w^1$, we get that
	\begin{align*}
	\|L^n g\|_s&\leq Qq^n(\|g\|_s+\|Pg\|_s)+ \|P L^n Pg\|_s.
	\end{align*}
	Letting $n= N\in\mathbb{N}$ large enough so that $Qq^N< 1$, we obtain \eqref{lasota-yorke} with $r= Qq^N$ and $R=\|P\|_{B_w\to B_s}(Qq^N+ \|L^N\|_w\|P\|_w)$.
\end{proof}
\begin{lemma}\label{corolariocontracaoimplicaly}
	Let $(B_w,\|\cdot\|_w)$ and $(B_s,\|\cdot\|_s)$ be Banach spaces, with $B_s\subset B_w$ where $\iota: B_s\to B_w$ is bounded and let $L: B_w\to B_w$ be a bounded linear operator s.t. $L(B_s)\subset B_s$ and $\sup_{n\in \mathbb{N}}\|L^n\|_w<\infty$. Assume that there exists a decomposition $B_w= B_w^0\oplus B_w^1$ into invariant subspaces $B_w^0$ and $B_w^1$ such that the inequality 
	
	\begin{equation}\label{contracao1}
	\forall n\in \mathbb{N}: \|L^n|_{B_s\cap B_w^0}\|_s\leq Qq^n,
	\end{equation} holds for some $q< 1$ and $Q\in\mathbb{R}$. Suppose that $L|_{B_w^1}$ is diagonalizable with unit spectrum, $B_w^1\subset B_s$ and $\dim B_w^1<\infty$. Then there exist $N\in\mathbb{N}$, $r< 1$ and $R\in\mathbb{R}$ such that \begin{equation}\label{lasota-yorke2}
	\|L^N g\|_s\leq r\|g\|_s + R\|g\|_w
	\end{equation} for all $g\in B_s$.
\end{lemma}
\begin{proof}
	Take a basis of eigenvectors $\{h_1,\ldots, h_d\}$ of $B_w^1$ and consider the linear isomorphism
	\begin{align*}
	A\colon B_w^1&\to\mathbb{R}^d\\
	\sum_{i= 1}^d\alpha_i h_i&\mapsto(\alpha_1,\ldots,\alpha_d).
	\end{align*}
	Denote by $\|\cdot\|_2$ the Euclidian norm in $\mathbb{R}^d$. Since $L|_{B_w^1}$, has unit spectrum and $\{h_1,\ldots, h_d\}$ is a basis of eigenvectors, we have that
	\begin{equation*}
	\forall n\in\mathbb{N}:\|AL^n g^1\|_2=\|Ag^1\|_2
	\end{equation*}
	for all $g^1\in B_w^1$. Therefore, for any $n\in\mathbb{N}$, $g^1\in B_w^1$ and $i,j\in\{w,s\}$,
	\begin{equation*}
	\frac 1{M_{i,j}}\leq\frac{\|L^n g^1\|_i}{\|g^1\|_j}\leq M_{i,j}:=\|A^{-1}\|_{i,2}\|A\|_{2,j}.
	\end{equation*}
	Thus, the projection $P: B_w\to B_w^1$ satisfies
	\begin{align*}
	\|Pg\|_s&\leq \|L^n Pg\|_wM_{w,s}\\
	&\leq \|L^n g\|_wM_{w,s} + \|L^n(g- Pg)\|_wM_{w,s}\\
	&\leq \sup_{n\in\mathbb{N}}\|L^n\|_w\|g\|_w M_{w,s}+ \|\iota\|_{+,*}\|L^n(g- Pg)\|_sM_{w,s}\\
	&\leq \sup_{n\in\mathbb{N}}\|L^n\|\|g\|_wM_{w,s}+ \|\iota\|_{+,*}Qq^n\|g\|_sM_{w,s}.
	\end{align*}
	Letting $n\to +\infty$,
	\begin{equation*}
	\|Pg\|_s\leq \sup_{n\in\mathbb{N}}\|L^n\|\|g\|_wM_{w,s}.
	\end{equation*}
	Thus, the projections $P$ is $(B_w, B_s)$-bounded and so $B_w$-bounded since $\iota$ is continuous. Hence, the hypothesis of Lemma \ref{contracaoimplicaly} are satisfied.
\end{proof}
	
\begin{theorem}\label{asewqtw} (Lasota-Yorke inequality) There exist $k\in \mathbb{N}$, $%
	0<\beta _{0}<1$ and $C>0$ such that, for all $g\in H_\zeta$, it holds 
	\begin{equation}
	|\func{P}_{f}^{k}g|_{\zeta}\leq \beta _{0}|g|_{\zeta}+C|g|_{\infty},  \label{LY1}
	\end{equation}where $|g|_\zeta := H_\zeta (g) + |g|_{\infty}$.
	
\end{theorem}

\begin{proof}	
Let $(B_w,\|\cdot\|_w)= (C_0, |\cdot|_\infty)$, $(B_s,\|\cdot\|_s)=(H_\zeta,|\cdot|_\zeta)$, $L=\func{P}_f$, and $f(\varphi)=\int\varphi dm_1$. The inclusion $\iota: B_s\to B_w$ is bounded and $L: B_w\to B_w$ is a bounded linear operator such that $L(B_s)\subset B_s$. $h= 1$ satisfies $Lh= h\neq 0$, $f: B_w\to\mathbb{R}$ is bounded and $f^{-1}(0)=\{\varphi\in H_0:\int\varphi dm_1= 0\}$. Take $B_w^0= f^{-1}(0)$ and $B_w^1=\mathbb{R} h$ in Lemma \ref{corolariocontracaoimplicaly}. Since Theorem \ref{loiub} implies condition \eqref{contracao1}, the hypothesis of the Lemma are satisfied.
\end{proof}

The following is a standard consequence of Theorem \ref{asewqtw} (and the fact that $1$ is a fixed point for $\func{P}_{f}$) that allows us to estimate
the behaviour of any given power of the transfer operator. 

\begin{corollary}\label{irytrtrte}
	There exist constants $B_3>0$, $C_2>0$ and $0<\beta_2<1$ such that for all $%
	g \in H_\zeta$, and all $n \geq 1$, it holds
	
	\begin{equation}
	|\func{P}_{f}^{n}g|_{\zeta} \leq B_3 \beta _2 ^n | g|_{\zeta} + C_2|g|_{\infty},
	\label{lasotaiiii}
	\end{equation}where $|g|_\zeta := H_\zeta (g) + |g|_{\infty}$.
\end{corollary}

\section{Weak and strong spaces\label{sec:spaces}}

\subsection{$L^{\infty}$-like spaces.}

Through this section, we construct some function spaces which are suitable
for the systems defined in section \ref{sec2}. The idea is to define spaces
of signed measures, where the norms are provided by disintegrating measures
along the stable foliation. Thus, a signed measure will be seen as a family
of measures on each leaf. For instance, a measure on the square with a
vertical foliation will be seen as a one parameter family (a path) of
measures on the interval (a stable leaf), where this identification will be
done by means of the Rokhlin's Disintegration Theorem. Finally, in the
vertical direction (on the leaves), we will consider a norm which is the
dual of the $\zeta$-H\"older norm and in the \textquotedblleft
horizontal\textquotedblright direction we will consider essentially the $%
L_{m_1}^{\infty}$ norm.

\subsubsection*{Rokhlin's Disintegration Theorem}

Now we briefly recall disintegration of measures.

Consider a probability space $(\Sigma,\mathcal{B}, \mu)$ and a partition $%
\Gamma$ of $\Sigma$ into measurable sets $\gamma \in \mathcal{B}$. Denote by $%
\pi : \Sigma \longrightarrow \Gamma$ the projection that associates to each
point $x \in M$ the element $\gamma _x$ of $\Gamma$ which contains $x$, i.e., 
$\pi(x) = \gamma _x$. Let $\widehat{\mathcal{B}}$ be the $\sigma$-algebra of 
$\Gamma$ provided by $\pi$. Precisely, a subset $\mathcal{Q} \subset \Gamma$
is measurable if, and only if, $\pi^{-1}(\mathcal{Q}) \in \mathcal{B}$. We
define the \textit{quotient} measure $\mu _x$ on $\Gamma$ by $\mu _x(%
\mathcal{Q})= \mu(\pi ^{-1}(\mathcal{Q}))$.

The proof of the following theorem can be found in \cite{Kva}, Theorem
5.1.11 (items a), b) and c)) and Proposition 5.1.7 (item d)).

\begin{theorem}
	(Rokhlin's Disintegration Theorem) Suppose that $\Sigma $ is a complete and
	separable metric space, $\Gamma $ is a measurable partition 
	of $\Sigma $ and $\mu $ is a probability on $\Sigma $. Then, $\mu $ admits a
	disintegration relative to $\Gamma $, i.e., a family $\{\mu _{\gamma
	}\}_{\gamma \in \Gamma }$ of probabilities on $\Sigma $ and a quotient
	measure $\mu _{x}$ as above, such that:
	
	\begin{enumerate}
		\item[(a)] $\mu _\gamma (\gamma)=1$ for $\mu _x$-a.e. $\gamma \in \Gamma$;
		
		\item[(b)] for all measurable set $E\subset \Sigma $ the function $\Gamma
		\longrightarrow \mathbb{R}$ defined by $\gamma \longmapsto \mu _{\gamma
		}(E), $ is measurable;
		
		\item[(c)] for all measurable set $E\subset \Sigma $, it holds $\mu (E)=\int 
		{\mu _{\gamma }(E)}d\mu _{x}(\gamma )$.

	\label{rok}
	\item [(d)] If the $\sigma $-algebra $\mathcal{B}$ on $\Sigma $ has a countable
	generator, then the disintegration is unique in the following sense. If $(\{\mu _{\gamma }^{\prime }\}_{\gamma \in \Gamma },\mu _{x})$ is another disintegration of the measure $\mu $ relative to $\Gamma $, then $\mu
	_{\gamma }=\mu _{\gamma }^{\prime }$, for $\mu _{x}$-almost every $\gamma
	\in \Gamma $.
	\end{enumerate}
\end{theorem}

\subsubsection{The $\mathcal{L}^{\infty}$ and $S^\infty$ spaces}

Let $\mathcal{SB}(\Sigma )$ be the space of Borel  signed measures on $\Sigma : = M \times K$. Given $\mu \in \mathcal{SB}(\Sigma )$ denote by $\mu ^{+}$ and $\mu ^{-}$
the positive and the negative parts of its Jordan decomposition, $\mu =\mu
^{+}-\mu ^{-}$ (see remark {\ref{ghtyhh}). Let $\pi _{x}:\Sigma
	\longrightarrow M$ \ be the projection defined by $\pi_x (x,y)=x$, denote
	by $\pi _{x\ast }:$}$\mathcal{SB}(\Sigma )\rightarrow \mathcal{SB}(M)${\
	the pushforward map associated to $\pi _{x}$. Denote by $\mathcal{AB}$ the
	set of signed measures $\mu \in \mathcal{SB}(\Sigma )$ such that its
	associated positive and negative marginal measures, $\pi _{x\ast }\mu ^{+}$
	and $\pi _{x\ast }\mu ^{-},$ are absolutely continuous with respect to $m_{1}$, i.e.,
	\begin{equation*}
	\mathcal{AB}=\{\mu \in \mathcal{SB}(\Sigma ):\pi _{x\ast }\mu ^{+}<<m_{1}\ \ 
	\mathnormal{and}\ \ \pi _{x\ast }\mu ^{-}<<m_{1}\}.  \label{thespace1}
	\end{equation*}%
}Given a \emph{probability measure} $\mu \in \mathcal{AB}$ on $\Sigma $,
Theorem \ref{rok} describes a disintegration $\left( \{\mu _{\gamma
}\}_{\gamma },\mu _{x}\right) $ along $\mathcal{F}^{s}$ by a family $\{\mu _{\gamma }\}_{\gamma }$ of probability measures
on the stable leaves\footnote{%
	In the following to simplify notations, when no confusion is possible we
	will indicate the generic leaf or its coordinate with $\gamma $.} and, since 
$\mu \in \mathcal{AB}$, $\mu _{x}$ can be identified with a non negative
marginal density $\phi _{x}:M\longrightarrow \mathbb{R}$, defined almost
everywhere, with $|\phi _{x}|_{1}=1$. \ For a general (non normalized)
positive measure $\mu \in \mathcal{AB}$ we can define its disintegration in
the same way. In this case, $\mu _{\gamma }$ are still probability measures, $%
\phi _{x}$ is still defined and $\ |\phi _{x}|_{1}=\mu (\Sigma )$.


\begin{definition}
	Let $\pi _{y}:\Sigma \longrightarrow K$ be the projection defined by $%
	\pi _{y}(x,y)=y$. Let $\gamma \in \mathcal{F}^{s}$, consider $\pi
	_{\gamma ,y}:\gamma \longrightarrow K$, the restriction of the map $\pi
	_{y}:\Sigma \longrightarrow K$ to the vertical leaf $\gamma $, and the
	associated pushforward map $\pi _{\gamma ,y\ast }$. Given a positive measure 
	$\mu \in \mathcal{AB}$ and its disintegration along the stable leaves $%
	\mathcal{F}^{s}$, $\left( \{\mu _{\gamma }\}_{\gamma },\mu _{x}=\phi
	_{x}m_{1}\right) $, we define the \textbf{restriction of $\mu $ on $\gamma $}
	and denote it by $\mu |_{\gamma }$ as the positive measure on $K$ (not
	on the leaf $\gamma $) defined, for all mensurable set $A\subset K$, as 
	\begin{equation*}
	\mu |_{\gamma }(A)=\pi _{\gamma ,y\ast }(\phi _{x}(\gamma )\mu _{\gamma
	})(A).
	\end{equation*}%
	For a given signed measure $\mu \in \mathcal{AB}$ and its Jordan
	decomposition $\mu =\mu ^{+}-\mu ^{-}$, define the \textbf{restriction of $%
		\mu $ on $\gamma $} by%
	\begin{equation*}
	\mu |_{\gamma }=\mu ^{+}|_{\gamma }-\mu ^{-}|_{\gamma }.
	\end{equation*}%
	\label{restrictionmeasure}
\end{definition}

\begin{remark}
	\label{ghtyhh}As proved in Appendix 2 of \cite {GLu}, the restriction $%
	\mu |_{\gamma }$ does not depend on the decomposition. Precisely, if $\mu
	=\mu _{1}-\mu _{2}$, where $\mu _{1}$ and $\mu _{2}$ are any positive
	measures, then $\mu |_{\gamma }=\mu _{1}|_{\gamma }-\mu _{2}|_{\gamma }$ $%
	m_{1}$-a.e. $\gamma \in M$. 
\end{remark}

Let $(X,d)$ be a compact metric space, $g:X\longrightarrow \mathbb{R}$ be a
$\zeta$-H\"older function and let $H_\zeta(g)$ be its best $\zeta$-H\"older's constant, i.e., 
\begin{equation}\label{lipsc}
\displaystyle{H_\zeta(g)=\sup_{x,y\in X,x\neq y}\left\{ \dfrac{|g(x)-g(y)|}{d(x,y)^\zeta}%
	\right\} }.
\end{equation}

In what follows, we generalize the definition of the Wasserstein-Kantorovich-like metric given in \cite{GLu} and \cite{GP}. This generalization allows us to obtain exponential decay of correlations on the set of $\zeta$-H\"older, instead of Lipschitz, functions. 
\begin{definition}
	Given two signed measures $\mu $ and $\nu $ on $X,$ we define a \textbf{\
		Wasserstein-Kantorovich-like} distance between $\mu $ and $\nu $ by 
	\begin{equation*}
	W_{1}^{\zeta}(\mu ,\nu )=\sup_{H_\zeta(g)\leq 1,|g|_{\infty }\leq 1}\left\vert \int {\
		g}d\mu -\int {g}d\nu \right\vert .
	\end{equation*}%
	\label{wasserstein}
\end{definition}

Since the constant $\zeta$ is fixed, from now on we denote%
\begin{equation}
||\mu ||_{W}:=W_{1}^{\zeta}(0,\mu ).  \label{WW}
\end{equation}%
As a matter of fact, $||\cdot ||_{W}$ defines a norm on the vector space of
signed measures defined on a compact metric space. It is worth to remark
that this norm is equivalent to the dual of the $\zeta$-H\"older norm. 

Other applications of this metric to obtain limit theorems can be seen in \cite{ben} and \cite {LiLu}. For instance, in \cite {ben} the author apply this metric to a more general case of shrinking fibers systems.

\begin{definition}
	Let $\mathcal{L}^{\infty }\subseteq \mathcal{AB}(\Sigma )$ be defined as%
	\begin{equation*}
	\mathcal{L}^{\infty }=\left\{ \mu \in \mathcal{AB}:\esssup ({W_{1}^{\zeta}(\mu
		^{+}|_{\gamma },\mu ^{-}|_{\gamma }))}<\infty \right\},
	\end{equation*}%
	where the essential supremum is taken over $M$ with respect to $m_{1}$.
	Define the function $||\cdot ||_{\infty }:\mathcal{L}^{\infty
	}\longrightarrow \mathbb{R}$ by%
	\begin{equation*}
	||\mu ||_{\infty }=\esssup ({W_{1}^{\zeta}(\mu ^{+}|_{\gamma },\mu ^{-}|_{\gamma
		}))}.
	\end{equation*}
\end{definition}

Finally, consider the following set of signed measures on $\Sigma $%
\begin{equation}\label{sinfi}
S^{\infty }=\left\{ \mu \in \mathcal{L}^{\infty };\phi _{x}\in
H_\zeta \right\},
\end{equation}%
and the function, $||\cdot ||_{S^{\infty }}:S^{\infty }\longrightarrow 
\mathbb{R}$, defined by%
\begin{equation*}
||\mu ||_{S^{\infty }}=|\phi _{x}|_{\zeta}+||\mu ||_{\infty }.
\end{equation*}

The proof of the next proposition is straightforward and can be found in 
\cite{L}.

\begin{proposition}
	$\left( \mathcal{L}^{\infty },||\cdot ||_{\infty }\right) $ and $\left(
	S^{\infty },||\cdot||_{S^{\infty }}\right) $ are normed vector spaces.
\end{proposition}


\section{The transfer operator associated to $F$}

In this section, we consider the transfer operator associated to skew product
maps as defined in Section 2, acting on our disintegrated measures spaces
defined in Section 3. For such transfer operators and measures we prove a
kind of Perron-Frobenius formula, which is somewhat similar to the one used
for one-dimensional maps.

Consider the pushforward map $\func{F}_{\ast }$ associated with $F$, defined
by 
\begin{equation*}
\lbrack \func{F}_{\ast }\mu ](E)=\mu (F^{-1}(E)),
\end{equation*}%
for each signed measure $\mu \in \mathcal{SB}(\Sigma )$ and for each
measurable set $E\subset \Sigma $, where $\Sigma:=M\times K$. When $\func{F}_{\ast }$ acts on
the vector space $\mathcal{SB}(\Sigma )$ or on suitable vector subspaces of
more regular measures, $\func{F}_{\ast }$ is a linear map. For this reason,
we also call it ``transfer operator associated to $F$".

\begin{lemma}
	\label{transformula}For every probability $\mu \in \mathcal{AB}$ disintegrated
	by $(\{\mu _{\gamma }\}_{\gamma },\phi _x)$, the disintegration $(\{(\func{%
		F}_{\ast }\mu )_{\gamma }\}_{\gamma },(\func{F}_{\ast }\mu )_{x})$ of the
	pushforward $\func{F}_{\ast }\mu $ \ satisfies the following relations%
	\begin{equation}
	(\func{F}_{\ast }\mu )_{x}=\func{P}_{f}(\phi _x)m_{1}  \label{1}
	\end{equation}
	and
	\begin{equation}
	(\func{F}_{\ast }\mu )_{\gamma }=\nu _{\gamma }:=\frac{1}{\func{P}_{f}(\phi
		_x)(\gamma )}\sum_{i=1}^{\deg(f)}{\frac{\phi _x}{|\det Df_{i}|}\circ
		f_{i}^{-1}(\gamma )\cdot \chi _{f_{i}(P_{i})}(\gamma )\cdot \func{F}_{\ast
		}\mu _{f_{i}^{-1}(\gamma )}}  \label{2}
	\end{equation}
	when $\func{P}_{f}(\phi _x)(\gamma )\neq 0$. Otherwise, if $\func{P}%
	_{f}(\phi _x)(\gamma )=0$, then $\nu _{\gamma }$ is the Lebesgue measure
	on $\gamma $ (the expression $\displaystyle{\frac{\phi _x}{|\det Df_{i}|}%
		\circ f_{i}^{-1}(\gamma )\cdot \frac{\chi _{f_{i}(P_{i})}(\gamma )}{\func{P}%
			_{f}(\phi _x)(\gamma )}\cdot \func{F}_{\ast }\mu _{f_{i}^{-1}(\gamma )}}$
	is understood to be zero outside $f_{i}(P_{i})$ for all $i=1,\cdots ,\deg(f)$).
	Here and above, $\chi _{A}$ is the characteristic function of the set $A$.
\end{lemma}

\begin{proof}
	By the uniqueness of the disintegration (see Theorem \ref{rok} ) it is enough to
	prove the following equation 
	\begin{equation*}
	\func{F}_{\ast }\mu (E)=\int_{M}{\nu _{\gamma }(E\cap \gamma )}\func{P}%
	_{f}(\phi _x)(\gamma )d m_1(\gamma) ,
	\end{equation*}%
	for a measurable set $E\subset \Sigma $. For this purpose, let us define the
	set $%
	B_{1}=\left\{ \gamma ;\func{P}_{f}(\phi _x)(\gamma
	)=0\right\}$. It is not hard to see that $\phi _x m_1 (f^{-1}(B_1))=0$.
	
	Using the change of variables $\gamma =f_{i}(\beta )$ and the definition of $%
	\nu _{\gamma }$ (see (\ref{2})), we have 
	\begin{align*}
		\int_{M}{\nu _{\gamma }(E\cap \gamma )}&\func{P}_{f}(\phi _x)(\gamma
		)dm_1(\gamma)\\
		 &=\int_{B_1 ^c}{\sum_{i=1}^{\deg(f)}{\frac{\phi _x}{|\det Df_{i}|}\circ
				f_{i}^{-1}(\gamma )\cdot \chi _{f_{i}(P_{i})}(\gamma )\cdot \func{F}_{\ast
				}\mu _{f_{i}^{-1}(\gamma )}}}dm_{1}(\gamma)\\
		&=\sum_{i=1}^{\deg(f)}{\int_{f_{i}(P_{i})\cap B_{1}^c}{\ {\frac{\phi _x}{|\det
						Df_{i}|}\circ f_{i}^{-1}(\gamma )\func{F}_{\ast }\mu _{f_{i}^{-1}(\gamma
						)}(E)}}}dm_{1}(\gamma ) \\ 
		&=\sum_{i=1}^{\deg(f)}{\int_{P_{i}\cap f_{i}^{-1}(B_1^c)}{\ {\phi _x(\beta )\mu
					_{\beta }(F^{-1}(E))}}}dm_{1}(\beta ) \\
		&={\int_{f^{-1}(B_1^c)}{\ {\phi _x(\beta )\mu _{\beta }(F^{-1}(E))}}}%
		dm_{1}(\beta ) \\
		&=\int_{M}{\ {\ \mu _{\beta }(F^{-1}(E))}}d\phi _xm_{1}(\beta ) \\
		&=\mu (F^{-1}(E)) \\
		&=\func{F}_{\ast }\mu (E).&\qedhere
	\end{align*}%
\end{proof}

As said in Remark \ref{ghtyhh}, the
restriction $\mu |_{\gamma }$ does not depend on the decomposition. Thus,
for each $\mu \in \mathcal{L}^{\infty}$, since $\func{F}_{\ast }\mu $ can be
decomposed as $\func{F}_{\ast }\mu =\func{F}_{\ast }(\mu ^{+})-\func{F}%
_{\ast }(\mu ^{-})$, we can apply the above Lemma to $\func{F}_{\ast }(\mu
^{+})$ and $\func{F}_{\ast }(\mu ^{-})$ to get the following.

\begin{proposition}
	\label{niceformulaab}Let $\gamma \in \mathcal{F}^{s}$ be a stable leaf. Let
	us define the map $F_{\gamma }:K\longrightarrow K$ by 
	\begin{equation}\label{ritiruwt}
	F_{\gamma }=\pi _{y}\circ F|_{\gamma }\circ \pi _{\gamma ,y}^{-1}.
	\end{equation}%
	Then, for each $\mu \in \mathcal{L}^{\infty}$ and for almost all $\gamma \in
	M $ (interpreted as the quotient space of leaves) it holds 
	\begin{equation}
	(\func{F}_{\ast }\mu )|_{\gamma }=\sum_{i=1}^{\deg(f)}{\func{F}%
		_{\gamma _i \ast }\mu |_{\gamma _i }\rho _i(\gamma _i)\chi _{f_{i}(P_{i})}(\gamma )}\ \ m_{1}%
	\mathnormal{-a.e.}\ \ \gamma \in M  \label{niceformulaa}
	\end{equation}%
	where $\func{F}_{\gamma_i \ast }$ is the pushforward map
	associated to $\func{F}_{\gamma_i}$, $\gamma _i = f_{i}^{-1}(\gamma )$ when $\gamma \in f_i (P_i)$ and $\rho_i(\gamma)= \dfrac{1}{|\det (f_i^{'}(\gamma))|}$, where $f_i = f|_{P_i}$.
\end{proposition}

Sometimes (see also Remark \ref{chkjg}) it will be convenient to use the following expression for $(\func{F}_{\ast } \mu )|_{\gamma }$:
\begin{corollary}\label{oierew}
	For each $\mu \in \mathcal{L}^{\infty}$ it holds 
	\begin{equation}
	(\func{F}_{\ast }\mu )|_{\gamma }=\sum_{i=1}^{\deg(f)}{\func{F}%
		_{\gamma _i \ast }\mu |_{\gamma _i }\rho _i(\gamma _i)}\ \ m_{1}%
	\mathnormal{-a.e.}\ \ \gamma \in M,  \label{niceformulaaareer}
	\end{equation}%
	where $\gamma _i$ is the $i$-th pre image of $\gamma$, $i=1,\cdots, \deg(f)$.
\end{corollary}

\section{Basic properties of the norms and convergence to equilibrium}

In this section, we show important properties of the norms and their
behaviour with respect to the transfer operator. In particular, we prove
that the $\mathcal{L}^{\infty}$ norm is weakly contracted. We also prove other properties, like a Lasota-Yorke inequality for the strong norm and exponential convergence to
equilibrium. All these properties will be used in next section to prove the
spectral gap for the transfer operator associated to the system $F:\Sigma
\rightarrow \Sigma $.

\begin{proposition}[The weak norm is weakly contracted by $\func{F}_{\ast }$]

	\label{l1} If $\mu \in \mathcal{L}^{\infty}$ then 
	\begin{equation*}
	||\func{F}_{\ast }\mu ||_{\infty}\leq ||\mu ||_{\infty}.
	\end{equation*}%
	\label{weakcontral11234}
\end{proposition}

In the proof of the proposition we will use the following lemma about the
behaviour of the $||\cdot ||_W$ norm (see equation (\ref{WW})) which says
that a contraction cannot increase the $||\cdot ||_W$ norm.

\begin{lemma}
	\label{niceformulaac} For every $\mu \in \mathcal{AB}$ and a stable leaf $%
	\gamma \in \mathcal{F}^{s}$, it holds 
	\begin{equation}
	||\func{F}_{\gamma \ast }\mu |_{\gamma }||_{W}\leq ||\mu |_{\gamma }||_{W},
	\label{weak1}
	\end{equation}%
	where $F_{\gamma }:K\longrightarrow K$ is defined in Proposition \ref%
	{niceformulaab} and $\func{F}_{\gamma \ast }$ is the associated pushforward
	map. Moreover, if $\mu $ is a probability measure on $K$, it holds 
	\begin{equation}
	||\func{F}_{\gamma \ast }{^{n}}\mu ||_{W}=||\mu ||_{W}=1,\ \ \forall \ \
	n\geq 1.  \label{simples}
	\end{equation}
\end{lemma}

\begin{proof}
	(of Lemma \ref{niceformulaac}) Indeed, since $F_{\gamma }$ is an $\alpha $%
	-contraction, if $|g|_{\infty }\leq 1$ and $H_\zeta(g)\leq 1$, the same holds for 
	$g\circ F_{\gamma }$. Since 
	\begin{equation*}
	\left\vert \int {g~}d\func{F}_{\gamma \ast }\mu |_{\gamma }\right\vert
	=\left\vert \int {g(F_{\gamma })~}d\mu |_{\gamma }\right\vert ,
	\end{equation*}%
	taking the supremum over $g$ \ such that $|g|_{\infty }\leq 1$ and $%
	H_\zeta(g)\leq 1$ we finish the proof of the inequality \ref{weak1}.
	
	In order to prove equation (\ref{simples}), consider a probability measure $%
	\mu $ on $K$ and a $\zeta$-H\"older function $g:K\longrightarrow \mathbb{R}$%
	, such that $||g||_{\infty }\leq 1$. We get immediately $|\int {g}d\mu |\leq
	||g||_{\infty }\leq 1$, which yields $||\mu ||_{W}\leq 1$. Considering $%
	g\equiv 1$, we get $||\mu ||_{W}=1$.
\end{proof}

\begin{proof}
	(of Proposition \ref{l1} )
	
	Lemma \ref{niceformulaac} and Corollary \ref%
	{oierew} yield 
	\begin{eqnarray*}
		||(\func{F}_{\ast }\mu)|_\gamma ||_{W} 
		&\leq &\sum_{i=1}^{\deg(f)}{||\func{F}%
			_{\gamma _i \ast }\mu |_{\gamma _i }\rho _i(\gamma _i)||_W} \\
		&\leq &\sum_{i=1}^{\deg(f)}{\rho _i(\gamma _i)||\mu |_{\gamma _i }||_W}  \\
		&\leq &||\mu ||_{\infty}\sum_{i=1}^{\deg(f)}{\rho _i(\gamma _i)} \\
		&=&||\mu ||_{\infty}.
	\end{eqnarray*}We finish the proof by taking the essential suppremum over $\gamma$.
\end{proof}

The following proposition shows a regularizing action of the transfer
operator with respect to the strong norm. Such inequalities are usually
called Lasota-Yorke or Doeblin-Fortet inequalities.

\begin{proposition}[Lasota-Yorke inequality for $S^{\infty}$] 
	There exist $A$, $B_{2}>0$ and $\lambda <1$ ($\lambda = \beta _2$ of Corollary \ref{irytrtrte}) such that, for all $\mu
	\in S^{1}$, it holds%
	\begin{equation}
	||\func{F}_{\ast }^{n}\mu ||_{S^{\infty}}\leq A\lambda ^{n}||\mu
	||_{S^{\infty}}+B_{2}||\mu ||_{\infty},\ \ \forall n\geq 1.  \label{xx}
	\end{equation}%
	\label{lasotaoscilation2}
\end{proposition}

\begin{proof}
	
	Firstly we recall that $\phi _{x}$ is the marginal density of the
	disintegration of $\mu $. Precisely, $\phi _{x}=\phi _{x}^{+}-\phi _{x}^{-}$%
	, where $\phi _{x}^{+}=\dfrac{d\pi _{x}^{\ast }\mu ^{+}}{dm_{1}}$ and $\phi
	_{x}^{-}=\dfrac{d\pi _{x}^{\ast }\mu ^{-}}{dm_{1}}$. By the definition of
	the Wasserstein norm, it follows that for every $\gamma $ it holds $||\mu
	|_{\gamma }||_{W}\geq \left| \int 1~d(\mu |_{\gamma })\right|=|\phi _{x}(\gamma )|$. Thus, 
	$|\phi _{x}|_{\infty}\leq ||\mu ||_{\infty}.$ By this last remark, equation (\ref%
	{lasotaiiii}) and Proposition \ref{l1} we have%
	\begin{eqnarray*}
		||\func{F}_{\ast }^{n}\mu ||_{S^{\infty}} &=&|\func{P}_{f}^{n}\phi _{x}|_{\zeta}+||%
		\func{F}_{\ast }^{n}\mu ||_{\infty} \\
		&\leq &B_{3}\beta _{2}^{n}|\phi _{x}|_{\zeta}+C_{2}|\phi _{x}|_{\infty}+||\mu ||_{\infty}
		\\
		&\leq &B_{3}\beta _{2}^{n}||\mu ||_{S^{\infty}}+(C_{2}+1)||\mu ||_{\infty}.
	\end{eqnarray*}%
	We finish the proof by setting $\lambda =\beta _{2}$, $A=B_{3}$ and $%
	B_{2}=C_{2}+1$.
\end{proof}

\subsection{Convergence to equilibrium}\label{invt}

Let $X$ be a compact metric space. Consider the space $\mathcal{SB}(X)$ of
signed Borel measures on $X$. \ In the following, we consider two further
normed vectors spaces of signed Borel measures on $X.$ The spaces $%
(B_{s},||~||_{s})\subseteq (B_{w},||~||_{w})\subseteq \mathcal{SB}(X)$ with
norms satisfying%
\begin{equation*}
||~||_{w}\leq ||~||_{s}.
\end{equation*}%
We say that a Markov operator
\begin{equation*}
	\text{L}:B_{w}\longrightarrow B_{w}
\end{equation*}
has convergence to equilibrium with speed at least $\Phi $ and with respect to
the norms $||\cdot ||_{s}$ and $||\cdot ||_{w}$, if for each $\mu \in 
\mathcal{V}_{s}$, where 
\begin{equation}
\mathcal{V}_{s}=\{\mu \in B_{s},\mu (X)=0\}  \label{vs}
\end{equation}%
is the space of zero-average measures, it holds 
\begin{equation*}
||\text{L}^{n}(\mu )||_{w}\leq \Phi (n)||\mu ||_{s},  \label{wwe}
\end{equation*}%
where $\Phi (n)\longrightarrow 0$ as $n\longrightarrow \infty $.

In this section, we prove that $F_{\ast }$ has exponential convergence to
equilibrium. This is weaker than spectral gap. However, the
spectral gap follows from the above Lasota-Yorke inequality and the
convergence to equilibrium. Before the main statements we need some
preliminary lemmata. The following is somewhat similar to Lemma \ref%
{niceformulaac} considering the behaviour of the $||\cdot ||_{W}$ norm after
a contraction. It gives a finer estimate for zero average measures and it is useful to estimate the behaviour of our $W$ norms under contractions.

\begin{lemma}\label{opsdas}
	For all signed measures $\mu $ on $K$ and for all $\gamma \in \mathcal{F}%
	^{s}$, it holds%
	\begin{equation*}
	||\func{F}_{\gamma \ast }\mu ||_{W}\leq \alpha^\zeta ||\mu ||_{W}+\mu (K)
	\end{equation*}%
	($\alpha $ is the rate of contraction of $G$, see \eqref{contracting1}). In
	particular, if $\mu (K)=0$ then%
	\begin{equation*}
	||\func{F}_{\gamma \ast }\mu ||_{W}\leq \alpha^\zeta ||\mu ||_{W}.
	\end{equation*}%
	\label{quasicontract}
\end{lemma}

\begin{proof}
	If $H_\zeta(g)\leq 1$ and $||g||_{\infty }\leq 1$, then $g\circ F_{\gamma }$ is $%
	\alpha^\zeta $-H\"older. Moreover, since $||g||_{\infty }\leq 1$, then $||g\circ
	F_{\gamma }-\theta ||_{\infty }\leq \alpha^\zeta $, for some $\theta $ such that $%
	|\theta |\leq 1$. Indeed, let $z\in K$ be such that $|g\circ F_{\gamma
	}(z)|\leq 1$, set $\theta =g\circ F_{\gamma }(z)$ and let $d_{2}$ be the
	Riemannian metric of $K$. Since $\diam(K)=1$, we have 
	\begin{equation*}
	|g\circ F_{\gamma }(y)-\theta |\leq \alpha^\zeta d_{2}(y,z)\leq \alpha^\zeta
	\end{equation*}%
	and consequently $||g\circ F_{\gamma }-\theta ||_{\infty }\leq \alpha^\zeta $.
	
	This implies
	
	\begin{align*}
	\left\vert \int_{K}{g}d\func{F}_{\gamma \ast }\mu \right\vert &
	=\left\vert \int_{K}{g\circ F_{\gamma }}d\mu \right\vert \\
	& \leq \left\vert \int_{K}{g\circ F_{\gamma }-\theta }d\mu \right\vert
	+\left\vert \int_{K}{\theta }d\mu \right\vert \\
	& =\alpha^\zeta \left\vert \int_{K}{\frac{g\circ F_{\gamma }-\theta }{\alpha^\zeta }}%
	d\mu \right\vert +|\theta ||\mu (K)|.
	\end{align*}%
	Taking the supremum over $g$ such that $|g|_{\infty }\leq 1$ and $%
	H_\zeta(g)\leq 1$, we have $||\func{F}_{\gamma \ast }\mu ||_{W}\leq \alpha^\zeta ||\mu
	||_{W}+\mu (K)$. In particular, if $\mu (K)=0$, we get the second
	part.
\end{proof}

Now we are ready to show a key estimate regarding the behaviour of our weak $%
|| \cdot ||_{\infty} $ norm for the systems defined at beginning of
Section \ref{sec2}.

\begin{proposition}
	\label{5.6} For every signed measure $\mu \in \mathcal{L}^{\infty}$, it holds 
	\begin{equation}
	||\func{F}_{\ast }\mu ||_{\infty}\leq \alpha ^\zeta ||\mu ||_{\infty}+|\phi
	_{x}|_{\infty}.  \label{abovv}
	\end{equation}
\end{proposition}

\begin{proof}
	Let $f_{i}$ be the branches of $f$, for all $i=1\cdots \deg(f)$. Applying Lemma %
	\ref{quasicontract} on the third line below, we have 
	\begin{eqnarray*}
		||(\func{F}_{\ast }\mu )|_{\gamma }||_{W} &=&\left\vert \left\vert
		\sum_{i=1}^{\deg(f)}\func{F}_{\gamma_{i}}\ast \mu
		|_{\gamma_{i}} \rho(\gamma _i)\right\vert \right\vert _{W}\\
		&\leq&
		\sum_{i=1}^{\deg(f)}\left\vert \left\vert\func{F}_{\gamma_{i}}\ast \mu
		|_{\gamma_{i}} \rho(\gamma _i)\right\vert \right\vert _{W} \\
		&\leq &\sum_{i=1}^{\deg(f)} ( \alpha^\zeta||\mu |_{\gamma_i}||_{W}+|\phi
		_{x}(\gamma_i)|)\rho (\gamma _i) \\
		&\leq &(\alpha^\zeta ||\mu ||_{\infty}+|\phi
		_{x}|_{\infty})\sum_{i=1}^{\deg(f)}\rho(\gamma_1)\\
		&= &\alpha^\zeta ||\mu ||_{\infty}+|\phi
		_{x}|_{\infty}.
	\end{eqnarray*}%
	Hence, by taking the supremum on $\gamma $, we finish the proof of the
	statement.
\end{proof}

Iterating (\ref{abovv}) we get the following corollary.

\begin{corollary}
	For every signed measure $\mu \in \mathcal{L}^{\infty}$ it holds 
	\begin{equation*}
	||\func{F}_{\ast }^{n}\mu ||_{\infty}\leq (\alpha ^\zeta) ^{n}||\mu ||_{\infty}+\overline{%
		\alpha }|\phi _{x}|_{\infty},
	\end{equation*}%
	where $\overline{\alpha }=\frac{1 }{1-\alpha^\zeta }$. \label{nicecoro}
\end{corollary}

Let us consider the set of zero average measures in $S^{\infty}$ defined by 
\begin{equation}
\mathcal{V}_{s}=\{\mu \in S^{\infty}:\mu (\Sigma )=0\}.  \label{mathV}
\end{equation}%
Note that, for all $\mu \in \mathcal{V}_{s}$ we have $\pi _{x\ast }\mu
(M)=0$. Moreover, since $\pi _{x\ast }\mu =\phi _{x}m_{1}$ ($\phi
_{x}=\phi _{x}^{+}-\phi _{x}^{-}$), we have $\displaystyle{\int_{M}{\phi
		_{x}}dm_{1}=0}$. This allows us to apply Theorem \ref{loiub} in the proof of
the next proposition.

\begin{proposition}[Exponential convergence to equilibrium]
	\label{5.8} There exist $D_{2}\in \mathbb{R}$ and $0<\beta _{1}<1$ such that
	for every signed measure $\mu \in \mathcal{V}_{s}$, it holds 
	\begin{equation*}
	||\func{F}_{\ast }^{n}\mu ||_{\infty}\leq D_{2}\beta _{1}^{n}||\mu ||_{S^{\infty}},
	\end{equation*}%
	for all $n\geq 1$, where $\beta _{1}=\max \{\sqrt{r},%
	\sqrt{\alpha^\zeta }\}$ and $D_{2}=(\sqrt{\alpha^\zeta }^{-1}+\overline{\alpha }D \sqrt{r}^{-1})$.\label{quasiquasiquasi}
\end{proposition}

\begin{proof}
	In this proof, to simplify the notation, we denote the constant $\alpha^\zeta$ just by $\alpha$.

	Given $\mu \in \mathcal{V}_s$ and denoting $\phi _{x}=\phi _{x}^{+}-\phi
	_{x}^{-}$, it holds that $\int {\phi }_{x}dm _1=0$. Moreover, Theorem \ref
	{loiub} yields $|\func{P}_{f}^{n}(\phi _{x})|_{\zeta}\leq Dr^{n}|\phi _{x}|_{\zeta}$
	for all $n\geq 1$, then (since ${|\phi_x|}_{\infty}\le {\|\mu\|}_{\infty}$) $|\func{P}_{f}^{n}(\phi _{x})|_{\zeta}\leq Dr^{n}||\mu
	||_{S^{\infty}}$, for all $n\geq 1$.
	
	Let $l$ and $0\leq d\leq 1$ be the coefficients of the division of $n$ by $2$%
	, i.e., $n=2l+d$. Thus, $l=\frac{n-d}{2}$ (by Proposition \ref%
	{weakcontral11234}, we have $||\func{F}_{\ast }^{n}\mu ||_{\infty}\leq ||\mu
	||_{\infty}$, for all $n$, and $||\mu ||_{\infty}\leq ||\mu ||_{S^{\infty}}$) and by
	Corollary \ref{nicecoro}, it holds (below, set $\beta _{1}=\max \{\sqrt{r},%
	\sqrt{\alpha^\zeta }\}$)
	
	\begin{eqnarray*}
		||\func{F}_{\ast }^{n}\mu ||_{\infty} &= &||\func{F}_{\ast }^{2l+d}\mu ||_{\infty}
		\\
		&\leq &(\alpha ^\zeta) ^{l}||\func{F}_{\ast }^{l+d}\mu ||_{\infty}+\overline{\alpha }%
		\left\vert \dfrac{d(\pi _{x\ast }(\func{F}_{\ast}^{l+d}\mu ))}{dm_{1}}%
		\right\vert _{\infty} \\
		&\leq &(\alpha^\zeta) ^{l}||\mu ||_{\infty}+\overline{\alpha }|\func{P}_{f}^{l}(\phi
		_{x})|_{\infty} \\
		&\leq &(\sqrt{\alpha^\zeta }^{-1}+\overline{\alpha }D \sqrt{r}^{-1})\beta _{1}^{n}||\mu ||_{S^{\infty}}
		\\
		&\leq &D_{2}\beta _{1}^{n}||\mu ||_{S^{\infty}},
	\end{eqnarray*}%
	where $D_{2}=(\sqrt{\alpha^\zeta }^{-1}+\overline{\alpha }D \sqrt{r}^{-1})$.
\end{proof}

The next lemma \ref{kjdhkskjfkjskdjf} ensures the existence and uniqueness of an $F$-invariant measure which projects on $m_1$. Since its proof is done by standard arguments (see \cite{AP}, for instance) we skip it. 

\begin{lemma}\label{kjdhkskjfkjskdjf}
	There exists an unique measure $\mu_0$ on $M \times K$ such that for every continuous function $\psi \in C^0 (M \times K)$ it holds

	\begin{equation}
	\lim {\int{\inf_{\gamma \times K} \psi \circ F^n }dm_1(\gamma)}= \lim {\int{\sup_{\gamma \times K} \psi \circ F^n}dm_1 (\gamma)}=\int {\psi}d\mu_0. 
	\end{equation}Moreover, the measure $\mu_0$ is $F$-invariant and $\pi_x{_\ast}\mu_0 = m_1$.
\end{lemma}

Now we present the proof of Theorem \ref{belongss} which says that the system has a unique invariant measure $\mu _{0}\in S^{\infty}$.

\begin{proof}{of Theorem \ref{belongss}}

	Let $\mu _{0}$ be the $F$-invariant measure 
	such that $\pi _{x\ast }\mu _{0}=m_1$ (which do exist by Lemma \ref{kjdhkskjfkjskdjf}), where $1$ is the unique $%
	f$-invariant density in $H_\zeta$. Suppose that $%
	g:K\longrightarrow \mathbb{R}$ is a $\zeta$-H\"older function such that $%
	|g|_{\infty }\leq 1$ and $H_\zeta(g)\leq 1$. Then, it holds $\left\vert \int {g}%
	d(\mu _{0}|_{\gamma })\right\vert \leq |g|_{\infty }\leq
	1$. Hence, $\mu _{0}\in \mathcal{L}^{\infty }$. Since, $\dfrac{\pi_{x*}\mu_0}{dm_1} \equiv 1$, we have $\mu_0 \in S^\infty$.
	
	The uniqueness follows directly from Proposition \ref{5.8}, since the difference between two probabilities ($\mu _1 - \mu_0$) is a zero average signed measure.
\end{proof}

\section{Spectral gap}\label{ierutierutiuyerit}

In this section, we prove a spectral gap result for the transfer operator applied to our strong space, Theorem \ref{spgapp}. For this, we will directly use the properties
proved in the previous section, and this will give a kind of constructive
proof. We remark that we cannot apply the traditional Hennion, or Ionescu-Tulcea and Marinescu's approach to our function spaces because there is no compact immersion of the strong space into the weak one. This comes from the fact that we are considering the same \textquotedblleft dual of
H\"older\textquotedblright distance (see Definition \ref{wasserstein}) in the contracting direction for both spaces. 

\begin{proof}{of Theorem \ref{spgapp}}

First, let us show there exist $0<\xi <1$ and $K_{1}>0$ such that, for all $%
	n\geq 1$, it holds 
	\begin{equation*}
	||\func{F}_{\ast }^{n}||_{{\mathcal{V}_{s}}\rightarrow {\mathcal{V}_{s}}%
	}\leq \xi ^{n}K_{1}  \label{quaselawww}
	\end{equation*}%
	where ${\mathcal{V}_{s}}$ is the zero average space defined in $($\ref{mathV}%
	$)$. Indeed, consider $\mu \in \mathcal{V}_{s}$ (see \eqref{mathV}) s.t. $%
	||\mu ||_{S^{\infty}}\leq 1$ and for a given $n\in \mathbb{N}$ let $m$ and $0\leq
	d\leq 1$ be the coefficients of the division of $n$ by $2$, i.e., $n=2m+d$.
	Thus $m=\frac{n-d}{2}$. By the Lasota-Yorke inequality (Proposition \ref%
	{lasotaoscilation2}) we have the uniform bound $||\func{F}_{\ast }^{n}\mu
	||_{S^{\infty}}\leq B_{2}+A$ for all $n\geq 1$. Moreover, by Propositions \ref%
	{quasiquasiquasi} and \ref{weakcontral11234} there is some $D_{2}$ such that
	it holds (below, let $\lambda _{0}$ be defined by $\lambda _{0}=\max \{\beta
	_{1},\lambda \}$)%
	\begin{eqnarray*}
		||\func{F}_{\ast }^{n}\mu ||_{S^{\infty}} &\leq &A\lambda ^{m}||\func{F}_{\ast
		}^{m+d}\mu ||_{S^{\infty}}+B_{2}||\func{F}_{\ast }^{m+d}\mu ||_{\infty} \\
		&\leq &\lambda ^{m}A(A+B_{2})+B_{2}||\func{F}_{\ast }^{m}\mu ||_{\infty} \\ 
		&\leq &\lambda ^{m}A(A+B_{2})+B_{2}D_{2}\beta _{1}^{m} \\
		&\leq &\lambda _{0}^{m}\left[ A(A+B_{2})+B_{2}D_{2}\right] \\
		&\leq &\lambda _{0}^{\frac{n-d}{2}}\left[ A(A+B_{2})+B_{2}D_{2}\right] \\
		&\leq &\left( \sqrt{\lambda _{0}}\right) ^{n}\left( \frac{1}{\lambda _{0}}%
		\right) ^{\frac{d}{2}}\left[ A(A+B_{2})+B_{2}D_{2}\right] \\
		&\leq&\xi ^{n}K_{1},
	\end{eqnarray*}%
	where $\xi =\sqrt{\lambda _{0}}$ and $K_{1}=\left( \frac{1}{\lambda _{0}}%
	\right) ^{\frac{1}{2}}\left[ A(A+B_{2})+B_{2}D_{2}\right] $. Thus, we arrive
	at 
	\begin{equation}
	||(\func{F}_{\ast }|_{_{\mathcal{V}_{s}}}){^{n}}||_{S^{\infty}\rightarrow
		S^{\infty}}\leq \xi ^{n}K_{1}.  \label{just}
	\end{equation}
	
	Now, recall that $\func{F}_{\ast }:S^{\infty}\longrightarrow S^{\infty}$ has a unique
	fixed point $\mu _{0}\in S^{\infty}$, which is a probability (see Theorem \ref%
	{belongss}). Consider the operator $\func{P}:S^{\infty}\longrightarrow \left[ \mu
	_{0}\right] $ ($\left[ \mu _{0}\right] $ is the space spanned by $\mu _{0}$%
	), defined by $\func{P}(\mu )=\mu (\Sigma )\mu _{0}$. By definition, $\func{P%
	}$ is a projection and $\dim\func{P}(S^{\infty})=1$. Define the operator%
	\begin{equation*}
	\func{S}:S^{\infty}\longrightarrow \mathcal{V}_{s},
	\end{equation*}%
	by%
	\begin{equation*}
	\func{S}(\mu )=\mu -\func{P}(\mu ),\ \ \ \mathnormal{\forall }\ \ \mu \in
	S^{\infty}.
	\end{equation*}%
	Thus, we set $\func{N}=\func{F}_{\ast }\circ \func{S}$ and observe that, by
	definition, $\func{P}\func{N}=\func{N}\func{P}=0$ and $\func{F}_{\ast }=%
	\func{P}+\func{N}$. Moreover, $\func{N}^{n}(\mu )=\func{F}_{\ast }{^{n}}(%
	\func{S}(\mu ))$ for all $n\geq 1$. Since $\func{S}$ is bounded and $\func{S}%
	(\mu )\in \mathcal{V}_{s}$, we get by (\ref{just}), $||\func{N}^{n}(\mu
	)||_{S^{\infty}}\leq \xi ^{n}K||\mu ||_{S^{\infty}}$, for all $n\geq 1$, where $%
	K=K_{1}||\func{S}||_{S^{\infty}\rightarrow S^{\infty}}$.
\end{proof}

\begin{remark}
	\label{quantitative2} The constant $\xi $ for the map $F$, found in Theorem %
	\ref{spgapp}, is directly related to the coefficients of the Lasota-Yorke
	inequality for the basis map (see Corollary \ref{irytrtrte}) and the rate of convergence to equilibrium of $F$ (see Proposition \ref{5.8}) found before. More precisely, $\xi =\max \{\sqrt{\lambda }%
	,\sqrt{\beta _{1}}\}$. We remark that, from the above proof we also have an
	explicit estimate for $K$ in the exponential convergence, while many
	classical approaches are not suitable for this.
\end{remark}


\subsection{Exponential Decay of Correlations}\label{decay} In this section, we present one of the standard consequences of spectral gap, Theorem \ref{decay1}. We will show how Theorem \ref{spgapp} implies an exponential rate of convergence for the limit $$\lim_{n\to\infty} {C_n(f,g)}=0,$$where $$C_n(f,g):=\left| \int{(g \circ F^n )  f}d\mu_0 - \int{g  }d\mu_0 \int{f  }d\mu_0 \right|,$$ $g: \Sigma \longrightarrow \mathbb{R} $ is a $\zeta$-H\"older function and $f \in \Theta _{\mu _0}$. The sets $\Theta _{\mu _0}$ are defined by $$\Theta _{\mu _0}:= \{ f: \Sigma \longrightarrow \mathbb{R}; f\mu_0 \in S^\infty\},$$ where the measure $f\mu_0$ is defined by $f\mu_0(E):=\int _E{f}d\mu_0$ for all measurable sets $E$.

\begin{proof}{of Theorem \ref{decay1}}
	
	Let $g: \Sigma \longrightarrow \mathbb{R} $ be a $\zeta$-H\"older function and $f \in \Theta _{\mu _0}$. By Theorem \ref{spgapp}, we have
	
	\begin{align*}
		\left| \int{(g \circ F^n )  f}\right.\mu_0 - \int{g  }d\mu_0 &\left.\int{f  }d\mu_0 \right| 
		=\left| \int{g  }d \func{F_*}{^n} (f\mu_0) - \int{g  }d\func{P}(f\mu_0) \right| 
		\\&\leq\left|\left|  \func{F_*}{^n} (f\mu_0) - \func{P}(f\mu_0) \right|\right|_W \max\{H_\zeta(g), ||g||_\infty\}
		\\&=\left|\left|  \func{N}{^n}(f\mu_0) \right|\right|_W \max\{H_\zeta(g), ||g||_\infty\}
		\\&\leq\left|\left|  \func{N}{^n}(f\mu_0) \right|\right|_{S^\infty} \max\{H_\zeta(g), ||g||_\infty\}
		\\&\leq||f \mu _0||_{S^{\infty}} K |g|_{\zeta}  \xi ^{n}.\qedhere
	\end{align*}
\end{proof}

In Theorem \ref{disisisii} we will see, under some further assumptions on the system, that the sets $ \Theta _{\mu _0} $ contain the set of $\zeta$-H\"older functions on $\Sigma$.

\section{H\"older measures}\label{hhhhhhhh}

In this section we will prove Theorem \ref{regg}. We suppose that $G$ satisfies the additional property (G2) stated below. Moreover, besides satisfying equation (\ref{kdljfhkdjfkasd}), the constant $L$ mentioned in (f1) and (f3) is also supposed to be close enough to $1$ such that $(\alpha \cdot L)^\zeta<1$ (or $\alpha$ is close enough to $0$). This is not restrictive for the aimed examples and it is clearly satisfied by Example \ref{sesprowerpo}.

\begin{enumerate}
	\item [(G2)] Let $P_1, \cdots, P_{\deg(f)}$ be the partition of $M$ given in Remark \ref{chkjg}. Suppose that  
	\begin{equation}\label{oityy}
	|G_i|_\zeta:= \sup _y\sup_{x_1, x_2 \in P_i} \dfrac{d_2(G(x_1,y), G(x_2,y))}{d_1(x_1,x_2)^\zeta}< \infty.
	\end{equation}
\end{enumerate}And denote by $|G|_\zeta$ the following constant
\begin{equation}\label{jdhfjdh}
|G|_\zeta := \max_{i=1, \cdots, s} \{|G_i|_\zeta\}.
\end{equation}

\begin{remark}
	The condition (G2) means that $G$ can be discontinuous on the sets $\partial P_i \times K$, for all $i=1, \cdots, \deg(f)$, where $\partial P_i$ denotes the boundary of $P_i$.
\end{remark}

We have seen that a positive measure on $M \times K$ can be disintegrated along the stable
leaves $\mathcal{F}^s$ in a way that we can see it as a family of positive measures on $M$, $\{\mu |_\gamma\}_{\gamma \in \mathcal{F}^s }$. Since there is a one-to-one correspondence between $\mathcal{F}^s$  and $M$, this defines a  path
in the metric space of positive measures ($\mathcal{SB}(K)$) defined on $K$, $M \longmapsto \mathcal{SB}(K)$, where $\mathcal{SB}(K)$ is endowed with the Wasserstein-Kantorovich-like metric (see definition \ref{wasserstein}). 
It will be convenient to use a functional notation and denote such a path by  
$\Gamma_{\mu } : M \longrightarrow \mathcal{SB}(K)$  defined almost everywhere by $\Gamma_{\mu } (\gamma) = \mu|_\gamma$, where $(\{\mu _{\gamma }\}_{\gamma \in M},\phi_{x})$ is some disintegration of $\mu$.
However, since such a disintegration is defined $\widehat{\mu}$-a.e. $\gamma \in M$, the path $\Gamma_\mu$ is not unique. For this reason we define more precisely $\Gamma_{\mu } $ as the class of almost everywhere equivalent paths corresponding to $\mu$.

\begin{definition}
	Consider a positive Borel measure $\mu$ on $M \times K$ and a disintegration  $\omega=(\{\mu _{\gamma }\}_{\gamma \in M},\phi
	_x)$, where $\{\mu _{\gamma }\}_{\gamma \in M }$ is a family of
	probabilities on $M \times K$ defined $\widehat{\mu}$-a.e. $\gamma \in M$ (where $\widehat{\mu} := \pi_x{_*}\mu=\phi _x m_1$) and $\phi
	_x:\Sigma_A^+\longrightarrow \mathbb{R}$ is a non-negative marginal density. Denote by $\Gamma_{\mu }$ the class of equivalent paths associated to $\mu$ 
	\begin{equation*}
	\Gamma_{\mu }=\{ \Gamma^\omega_{\mu }\}_\omega,
	\end{equation*}
	where $\omega$ ranges on all the possible disintegrations of $\mu$ and $\Gamma^\omega_{\mu }: M\longrightarrow \mathcal{SB}(K)$ is the map associated to a given disintegration, $\omega$:
	$$\Gamma^\omega_{\mu }(\gamma )=\mu |_{\gamma } = \pi _{\gamma, y} ^\ast \phi _x
	(\gamma)\mu _\gamma .$$
\end{definition}Let us call the set on which $\Gamma_{\mu }^\omega $ is defined by $I_{\Gamma_{\mu }^\omega } \left( \subset M\right)$.



\begin{definition}For a given $0<\zeta <1$, a disintegration $\omega$ of $\mu$  and its functional representation $\Gamma_{\mu }^\omega $ we define the \textbf{$\zeta$-H\"older constant of $\mu$ associated to $\omega$} by

	\begin{equation}\label{Lips1}
	|\mu|_\zeta ^\omega := \esssup _{\gamma_1, \gamma_2 \in I_{\Gamma_{\mu }^\omega}} \left\{ \dfrac{||\mu|_{\gamma _1}- \mu|_{\gamma _2}||_W}{d_1 (\gamma _1, \gamma _2)^\zeta}\right\}.
	\end{equation}Finally, we define the \textbf{$\zeta$-H\"older constant} of the positive measure $\mu$ by

	\begin{equation}\label{Lips2}
	|\mu|_\zeta :=\displaystyle{\inf_{ \Gamma_{\mu }^\omega \in \Gamma_{\mu } }\{|\mu|_\zeta ^\omega\}}.
	\end{equation}
	
	\label{Lips3}
\end{definition}

\begin{remark}
	When no confusion is possible, to simplify the notation, we denote $\Gamma_{\mu }^\omega (\gamma )$ just by $\mu |_{\gamma } $.
\end{remark}

\begin{definition}
	From the Definition \ref{Lips3} we define the set of the $\zeta$-H\"older positive measures $\mathcal{H} _\zeta^{+}$ as
	
	\begin{equation}
	\mathcal{H} _\zeta^{+}=\{\mu \in \mathcal{AB}:\mu \geq 0,|\mu |_\zeta <\infty \}.
	\end{equation}
\end{definition}

For the next lemma, for a given path, $\Gamma _\mu$ which represents the measure $\mu$, we define for each $\gamma \in I_{\Gamma_{\mu }^\omega }\subset M$, the map

\begin{equation}
\mu _F(\gamma) := \func{F_\gamma }_*\mu|_\gamma,
\end{equation}where $F_\gamma :K \longrightarrow K$ is defined as

\begin{equation}\label{poier}
F_\gamma (y) = \pi_y \circ F \circ {(\pi _y|_\gamma)} ^{-1}(y)
\end{equation}and $\pi_y : M\times K \longrightarrow  K$ is the projection $\pi_y(x,y)=y$.

\begin{lemma}\label{apppoas}
	Suppose that $F:\Sigma \longrightarrow \Sigma$ satisfies (G1) and (G2). Then, for all $\mu \in \mathcal{H} _\zeta^{+} $ which satisfy $\phi _x = 1$ $m_1$-a.e., it holds $$||\func{F}%
	_{x  \ast }\mu |_{x  } - \func{F}%
	_{y \ast }\mu |_{y  }||_W \leq \alpha^\zeta |\mu|_\zeta  d_1(x, y)^\zeta  + |G|_\zeta d_1(x, y)^\zeta ||\mu ||_\infty,$$ for all $x,y \in P_i$ and all $i=1, \cdots, \deg(f)$.
\end{lemma}

\begin{proof}

	Since $(\mu|_x - \mu|_y)(K)=0$ ($\phi _x = 1$ $m_1$-a.e.), by Lemma \ref{opsdas}, it holds
	\begin{eqnarray*}
		||\func{F}%
		_{x  \ast }\mu |_{x  } - \func{F}%
		_{y \ast }\mu |_{y  }||_W &\leq & ||\func{F}%
		_{x  \ast }\mu |_{x  } - \func{F}%
		_{x \ast }\mu |_{y  }||_W + ||\func{F}%
		_{x  \ast }\mu |_{y  } - \func{F}%
		_{y \ast }\mu |_{y  }||_W
		\\&\leq & \alpha^\zeta||\mu |_{x  } - \mu |_{y }||_W + ||\func{F}%
		_{x  \ast }\mu |_{y  } - \func{F}%
		_{y \ast }\mu |_{y  }||_W
		\\&\leq & \alpha^\zeta |\mu|_\zeta d_1(x,y)^\zeta + ||\func{F}%
		_{x  \ast }\mu |_{y  } - \func{F}%
		_{y \ast }\mu |_{y  }||_W.
	\end{eqnarray*}Let us estimate the second summand $||\func{F}%
	_{x  \ast }\mu |_{y  } - \func{F}%
	_{y \ast }\mu |_{y  }||_W$. To do it, let $g:K \longrightarrow \mathbb{R}$ be a $\zeta$-H\"older function s.t. $H_\zeta(g), |g|_\infty \leq 1$. By equation (\ref{poier}), we get 
	
	\begin{align*}
	\begin{split}
		\left|\int gd(\func{F}_{x\ast}\mu|_y)-\int gd(\func{F}_{y\ast}\mu|_y) \right|&=\left|\int\!{g(G(x,z))}d(\mu|_y)(z)\right.\\
		&\qquad\left.-\int\!{g(G(y,z))}d(\mu|_y)(z) \right|
		\end{split}
		\\&\leq\int{\left|G(x,z)-G(y,z)\right|}d(\mu|_y)(z)
		\\&\leq|G|_\zeta d_1(x,y)^\zeta \int{1}d(\mu|_y)(z) 
		\\&\leq|G|_\zeta d_1(x,y)^\zeta ||\mu|_y||_W.
	\end{align*}Thus, taking the supremum over $g$ and the essential supremum over $y$, we get 
	
	\begin{equation*}
	||\func{F}%
	_{x  \ast }\mu |_{y  } - \func{F}%
	_{y \ast }\mu |_{y  }||_W \leq|G|_\zeta d_1(x,y)^\zeta ||\mu||_\infty.\qedhere
	\end{equation*}
	
\end{proof}

For the next proposition and henceforth, for a given path $\Gamma _\mu ^\omega \in \Gamma_{ \mu }$ (associated with the disintegration $\omega = (\{\mu _\gamma\}_\gamma, \phi _x)$, of $\mu$), unless written otherwise, we consider the particular path $\Gamma_{\func{F_*}\mu} ^\omega \in \Gamma_{\func{F_*}\mu}$ defined by the Corollary \ref{oierew}, by the expression

\begin{equation}
\Gamma_{\func{F_*}\mu} ^\omega (\gamma)=\sum_{i=1}^{\deg(f)}{\func{F}%
	_{\gamma _i \ast }\Gamma _\mu ^\omega (\gamma_i)\rho _i(\gamma _i)}\ \ m_{1}%
\mathnormal{-a.e.}\ \ \gamma \in M.  \label{niceformulaaareer}
\end{equation}Recall that $\Gamma_{\mu} ^\omega (\gamma) = \mu|_\gamma:= \pi_{y*}(\phi_{x}(\gamma)\mu _\gamma)$ and in particular $\Gamma_{\func{F_*}\mu} ^\omega (\gamma) = (\func{F_*}\mu)|_\gamma = \pi_{y*}(\func{P}_f\phi_x(\gamma)\mu _\gamma)$, where $\phi_x = \dfrac{d \pi _{x*} \mu}{dm_1}$ and $\func{P}_f$ is the Perron-Frobenius operator of $f$.
\begin{proposition}\label{iuaswdas}
	If $F:\Sigma \longrightarrow \Sigma$ satisfies (f1), (f2), (f3), (G1), (G2) and $\alpha \cdot L^\zeta<1$, then there exist $0<\beta<1$ and $D >0$, such that for all $\mu \in \mathcal{H} _\zeta^{+} $ which satisfy $\phi _x = 1$ $m_1$-a.e. and for all $\Gamma ^\omega _\mu \in \Gamma _\mu$, it holds $$|\Gamma_{\func{F_*} } ^\omega\mu|_{\zeta}  \leq \beta |\Gamma_{\mu}^\omega|_\zeta + D||\mu||_\infty,$$ for $\beta:= (\alpha L)^\zeta$ and $D:=\{\epsilon _\rho L^\zeta + |G|_ \zeta L^\zeta\}$.
\end{proposition}

\begin{proof}  
	
	In this proof, we denote a leaf by $x$ and $y$, instead of $\gamma$. We assume that the pre-images $x_i$ and $y_i$ belong to the same atom $P_i$ for all $i=1, \cdots, \deg(f)$, stated in (G2). This holds for $m_1$-a.e. $x,y \in M$. Moreover, $\sum_{i=1}^{\deg(f)}{\rho _i(y _i)}=1$, since $m_1$ is $f$-invariant. Note that
	\begin{eqnarray*}
		(\func{F_*}\mu)|_{x} - (\func{F_*}\mu)|_{y}&=&\sum_{i=1}^{\deg(f)}{\func{F}%
			_{x _i \ast }\mu |_{x _i }\rho _i(x _i)} - \sum_{i=1}^{\deg(f)}{\func{F}%
			_{y _i \ast }\mu |_{y _i }\rho _i(y _i)}
		\\&=&\sum_{i=1}^{\deg(f)}{\func{F}%
			_{x _i \ast }\mu |_{x _i }\rho _i(x _i)} - \sum_{i=1}^{\deg(f)}{\func{F}%
			_{x _i \ast }\mu |_{x _i }\rho _i(y _i)}
		\\&+&\sum_{i=1}^{\deg(f)}{\func{F}%
			_{x _i \ast }\mu |_{x _i }\rho _i(y _i)} - \sum_{i=1}^{\deg(f)}{\func{F}%
			_{y _i \ast }\mu |_{y _i }\rho _i(y _i)}.
	\end{eqnarray*}Thus, by Lemma \ref{niceformulaac}, we have 
	
	\begin{equation}\label{skdjghdjfjkhd}
	||(\func{F_*}\mu)|_{x} - (\func{F_*}\mu)|_{y}||_W \leq \func{I_1} + \func{I_2},
	\end{equation}where
	\begin{equation}\label{utrp}
	\func{I_1}= \sum_{i=1}^{\deg(f)}{||\mu |_{x _i } ||_W  |\rho _i(x _i) - \rho _i(y _i)|}
	\end{equation}and
	\begin{equation}\label{puipu}
	\func{I_2}= \sum_{i=1}^{\deg(f)}{|\rho _i(y _i)| ||\func{F}%
		_{x _i \ast }\mu |_{x _i } - \func{F}%
		_{y _i \ast }\mu |_{y _i }||_W}.
	\end{equation}Now, we estimate $\func{I_1}$ and $\func{I_2}$.

	By (f3), we have
	\begin{eqnarray*}
		\func I_1 &\leq& \esssup _x ||\mu |_{x } ||_W\sum_{i=1}^{\deg(f)}{  |\rho _i(x _i) - \rho _i(y _i)|}
		\\&\leq& ||\mu  ||_\infty  \sum_{i=1}^{\deg(f)}{H_\zeta(\rho) d_1(x_i,y_i)^\zeta}
		\\&\leq& ||\mu  ||_\infty  \sum_{i=1}^{\deg(f)}{\epsilon _\rho \inf{\rho} d_1(x_i,y_i)^\zeta}
		\\&\leq&  ||\mu  ||_\infty \epsilon _\rho \sum_{i=1}^{\deg(f)}{ \inf{\rho} d_1(x_i,y_i)^\zeta}
		\\&\leq&  ||\mu  ||_\infty \epsilon _\rho \sum_{i=1}^{\deg(f)}{ \rho (x_i) d_1(x_i,y_i)^\zeta}
		\\&\leq&  ||\mu  ||_\infty \epsilon _\rho L ^\zeta d_1(x,y)^\zeta \sum_{i=1}^{\deg(f)}{ \rho (x_i) }
		\\&\leq&  ||\mu  ||_\infty \epsilon _\rho L ^\zeta d_1(x,y)^\zeta.
	\end{eqnarray*}Thus
	
	\begin{equation}\label{glhkgf}
	\func I_1 \leq ||\mu  ||_\infty \epsilon _\rho L ^\zeta d_1(x,y)^\zeta.
	\end{equation}

	Let us estimate $I_2$. Since $x_i, y_i \in P_i$, Lemma \ref{apppoas} yields
	\begin{eqnarray*}
		\func{I_2}&=& \sum_{i=1}^{\deg(f)}{|\rho _i(y _i)| ||\func{F}%
			_{x _i \ast }\mu |_{x _i } - \func{F}%
			_{y _i \ast }\mu |_{y _i }||_W}
		\\&\leq& \sum_{i=1}^{\deg(f)}{\rho _i(y _i)\left(\alpha ^\zeta |\mu|_\zeta d_1(x_i, y_i)^\zeta + |G|_\zeta d_1(x_i, y_i)^\zeta ||\mu |_{y_i}||_W\right)}
		\\&\leq& \left(\alpha ^\zeta |\mu|_\zeta L^\zeta d_1(x, y)^\zeta  + |G|_\zeta L^\zeta d_1(x, y)^\zeta ||\mu ||_\infty \right)\sum_{i=1}^{\deg(f)}{\rho _i(y _i)}
		\\&=& \alpha ^\zeta |\mu|_\zeta L^\zeta d_1(x, y)^\zeta  + |G|_\zeta L^\zeta d_1(x, y)^\zeta ||\mu ||_\infty. 
	\end{eqnarray*}Hence 
	\begin{equation}\label{esti2}
	\func{I_2} \leq \alpha ^\zeta |\mu|_\zeta L^\zeta d_1(x, y)^\zeta  + |G|_\zeta L^\zeta d_1(x, y)^\zeta ||\mu ||_\infty.
	\end{equation}
	
	By equations (\ref{skdjghdjfjkhd}), (\ref{glhkgf}) and (\ref{esti2}), we have
	
	\begin{equation*}
	|\Gamma_{\func{F_*}\mu}^\omega|_{\zeta}  \leq \beta |\Gamma_{\mu}^\omega|_\zeta + D||\mu||_\infty,
	\end{equation*}$\beta:= (\alpha L)^\zeta$ and $D:=\{\epsilon _\rho L^\zeta + |G|_ \zeta L^\zeta\}$.

\end{proof}

Iterating the inequality $|\Gamma_{\func{F_*}\mu}^\omega|_{\zeta}  \leq \beta |\Gamma_{\mu}^\omega|_\zeta + D||\mu||_\infty$ obtained in Proposition \ref{iuaswdas}, a standard computation yields the next result, the proof of which is omitted.

\begin{corollary}\label{kjdfhkkhfdjfh}
	Suppose that $F:\Sigma \longrightarrow \Sigma$ satisfies (f1), (f2), (f3), (G1), (G2) and $(\alpha \cdot L)^\zeta<1$. Then, for all $\mu \in \mathcal{H} _\zeta^{+} $ which satisfy $\phi _x = 1$ $m_1$-a.e. and $||\func{F_*}\mu||_\infty \leq ||\mu||_\infty$, it holds 
	\begin{equation}\label{erkjwr}
	|\Gamma_{\func{F_*}^n\mu}^\omega|_{\zeta}  \leq \beta^n |\Gamma _\mu^\omega|_\zeta + \dfrac{D}{1-\beta}||\mu||_\infty,
	\end{equation}
	for all $n\geq 1$, where $\beta$ and $D$ are from Proposition \ref{iuaswdas}.
\end{corollary}

\begin{remark}\label{kjedhkfjhksjdf}
	Taking the infimum over all paths $\Gamma_{ \mu } ^\omega  \in \Gamma_{ \mu }$ and all $\Gamma_{\func{F_*}^n\mu}^\omega  \in \Gamma_{\func{F_*}^n\mu}$ on both sides of inequality (\ref{erkjwr}), we get 
	
	\begin{equation}\label{fljghlfjdgkdg}
	|\func{F_*}^n\mu|_{\zeta}  \leq \beta^n |\mu|_\zeta + \dfrac{D}{1-\beta}||\mu||_\infty. 
	\end{equation}The above Equation (\ref{fljghlfjdgkdg}) will give a uniform bound (see the proof of Theorem \ref{riirorpdf}) for the H\"older's constant of the measure $\func {F_*}^{n} m$, for all $n$, where $m$ is defined as the product $m=m_1 \times \nu$, for a fixed probability measure $\nu$ on $K$. The uniform bound will be useful later on (see Theorem \ref{disisisii}).
	
\end{remark}

\begin{remark}\label{riirorpdf}
	Consider the probability measure $m$ defined in Remark \ref{kjedhkfjhksjdf}, i.e., $m=m_1 \times \nu$, where $\nu$ is a given probability measure on $K$ and $m_1$ is the $f$-invariant measure fixed in the subsection \ref{hf}. Besides that, consider its trivial disintegration $\omega_0 =(\{m_{\gamma}  \}_{\gamma}, \phi_x)$, given by $m_\gamma = \func{\pi _{y,\gamma}^{-1}{_*}}\nu$, for all $\gamma$ and $\phi _x \equiv 1$. According to this definition, it holds that 
	\begin{equation*}
	m|_\gamma = \nu, \ \ \forall \ \gamma.
	\end{equation*}In other words, the path $\Gamma ^{\omega _0}_m$ is constant: $\Gamma ^{\omega _0}_m (\gamma)= \nu$ for all $\gamma$.  Moreover, for each $n \in \mathbb{N}$, let $\omega_n$ be the particular disintegration of the measure $\func{F{_\ast }}^nm$ defined from $\omega_0$ as an application of Lemma \ref{transformula}, and consider the path $\Gamma^{\omega_{n}}_{\func{F{_\ast }}^n m}$ associated with this disintegration. By Proposition \ref{niceformulaab}, we have

	\begin{equation}
	\Gamma^{\omega_{n}}_{\func{F{_\ast }} ^n m} (\gamma)  =\sum_{i=1}^{q}{\dfrac{\func{F^n%
				_{f_{i}^{-n}(\gamma )}}_{\ast  }\nu}{|\det Df^n_{i}\circ f_{i}^{-n}(\gamma ))|}\chi _{f^n_i(P _{i})}(\gamma )}\ \ m_1-\hbox{a.e} \ \ \gamma \in M,  \label{niceformulaaw}
	\end{equation}where $P_i$, $i=1, \cdots, q=q(n)$, ranges over the partition $\mathcal{P}^{(n)}$ defined in the following way: for all $n \geq 1$, let $\mathcal{P}^{(n)}$ be the partition of $I$ s.t. $\mathcal{P}^{(n)}(x) = \mathcal{P}^{(n)}(y)$ if and only if $\mathcal{P}^{(1)}(f^j (x)) = \mathcal{P}^{(1)}(f^j(y))$ for all $j = 0, \cdots , n-1$, where $\mathcal{P}^{(1)} = \mathcal{P}$ (see remark \ref{chkjg}). This path will be used in the proof of the next proposition.

\end{remark}

For the next result, we recall that, by Theorem \ref{belongss}, $F$ has a unique invariant measure $\mu _{0}\in S^{\infty }$. We will prove that $\mu_0$ has a regular disintegration in a way that $\mu _0 \in \mathcal{H}_\zeta$ (similar results are presented in \cite{BM}, \cite{GLu}, \cite{GP} and \cite{LiLu}, for others sort of systems). This will be used to prove exponential decay of correlations over the set of $\zeta$-H\"older functions.

\begin{proof}{of Theorem \ref{regg}}

	Consider the path $\Gamma^{\omega_n}_{\func{F{_\ast }^n}}m$, defined in Remark \ref{riirorpdf},  which represents the measure $\func{F{_\ast }}^nm$.

	
	According to Theorem \ref{belongss}, let $\mu _{0}\in S^\infty$ be the unique $F$-invariant probability measure in $S^\infty$. Consider the measure $m$, defined in Remark \ref{riirorpdf} and its iterates $\func{F{_\ast }}^n%
	(m)$. By Theorem \ref{spgapp}, these iterates converge to $\mu _{0}$
	in $\mathcal{L}^{\infty }$. It implies that the sequence $\{\Gamma_{\func{F{_\ast }}^n(m)} ^{\omega _n}\}_{n}$ converges $m_1$-a.e. to $\Gamma_{\mu _{0}}^\omega\in \Gamma_{\mu_0 }$ (in $\mathcal{SB}(K)$ with respect to the metric defined in definition \ref{wasserstein}),  where $\Gamma_{\mu _{0}}^\omega$ is a path given by the Rokhlin Disintegration
	Theorem and $\{\Gamma_{\func{F{_\ast }}^n(m)} ^{\omega_n}\}_{n}$ is given by equation (\ref{niceformulaaw}). It implies that $\{\Gamma_{\func{F{_\ast }}^n(m)} ^{\omega_n}\}_{n}$ converges pointwise to $\Gamma_{\mu _{0}}^\omega$ on a full measure set $\widehat{M}\subset M$. Let us denote $%
	\Gamma_{n}:=\Gamma^{\omega_n}_{\func{F{_\ast }}^n(m)}|_{%
		\widehat{M}}$ and $\Gamma:=\Gamma^\omega _{\mu _{0}}|_{\widehat{M}}$. Since $\{\Gamma_{n} \}_n $ converges pointwise to $\Gamma$, it holds $|\Gamma_{n}|_\zeta \longrightarrow |\Gamma|_\zeta$ as $n \rightarrow \infty$. Indeed, let $x,y \in \widehat{M}$. Then,
	
	\begin{eqnarray*}
		\lim _{n \longrightarrow \infty} {\dfrac{||\Gamma_n (x) - \Gamma _n(y)||_W}{d_1(x,y)^\zeta}} &= & \dfrac{||\Gamma (x) - \Gamma (y)||_W}{d_1(x,y)^\zeta}.
	\end{eqnarray*} On the other hand, by Corollary \ref{kjdfhkkhfdjfh}, the argument of the left hand side is bounded by $|\Gamma_n|_\zeta \leq  \dfrac{D}{1-\beta}$  for all $n\geq 1$. Then, 
	\begin{eqnarray*}
		\dfrac{||\Gamma (x) - \Gamma (y)||_W}{d_1(x,y)^\zeta}&\leq & \dfrac{D}{1-\beta}.
	\end{eqnarray*} Thus, $|\Gamma^\omega_{\mu _0}|_\zeta \leq\dfrac{D}{1-\beta}$ and taking the infimum we get $|\mu _0|_\zeta \leq\dfrac{D}{1-\beta}$.

\end{proof}

\section{From a Space of Measures to a Space of Functions\label{last123}}

In Section \ref{decay}, we proved that systems $F:\Sigma \longrightarrow \Sigma$ satisfying (f1), (f2), (f3), (G1), (G2) and $(\alpha \cdot L)^\zeta<1$ have exponential decay of correlations for observables $f \in \Theta _{\mu _0} ^1$ and H\"older ones. In this section, we prove that $\Theta _{\mu _0} ^1$ contains the set of $\zeta$-H\"older functions, Theorem \ref{disisisii}.

Denote the space of the $\zeta$-H\"older functions, $g:\Sigma\longrightarrow \mathbb{R}$, by $\ho_\zeta(\Sigma)$.
As a consequence of Theorem \ref{regg}, next Theorem \ref{disisisii} yields $\ho_\zeta(\Sigma) \subset \Theta _{\mu_0} $ (defined in subsection \ref{decay}). In order to prove it, we need the next Lemma \ref{hdgfghddsfg} on disintegration of absolutely continuous measures with respect to a measure $\mu_0 \in \mathcal{AB}$, the proof of which is postponed to the Appendix 1 (section \ref{disint}).

\begin{lemma}\label{hdgfghddsfg}
	Let $(\{\mu_{0, \gamma}\}_\gamma, \phi_x)$ be the disintegration of $\mu _0$ along the partition $\mathcal{F}^s:=\{\{\gamma\} \times K: \gamma \in M\}$, and for a $\mu_0$ integrable function $h:M \times K \longrightarrow \mathbb{R}$, denote by $\nu$ the measure $\nu:=h\mu_0$ ($ h\mu_0(E) := \int _E {h}d\mu _0$). If $(\{\nu_{ \gamma}\}_\gamma, \widehat{\nu} )$ is the disintegration of $\nu$, where $\widehat{\nu}:=\pi_x{_*} \nu$, then $\widehat{\nu} \ll m_1$ and $\nu _\gamma \ll \mu_{0, \gamma}$. Moreover,
	\begin{equation}\label{fjgh}
	\dfrac{d\widehat{\nu}}{dm_1}:=\overline{h}(\gamma)=\int_{M}{h(\gamma, y)}d(\mu_0|_\gamma),
	\end{equation} and for $\widehat{\nu}$-a.e. $\gamma \in M$ and all $y\in K$,
	
\begin{equation}\label{gdfgdgf}
	\dfrac{d\nu _{ \gamma}}{d\mu _{0, \gamma}}(y) =
	\begin{cases}
		\dfrac{h|_\gamma (y)}{\int{h|_\gamma(y)}d\mu_{0,\gamma}(y)} ,&\text{if } \gamma \in B ^c \\
		& \\
		\qquad \qquad 0,&\text{if } \gamma \in B,
	\end{cases}
\end{equation}
	where $B :=  \overline{h} ^{-1}(0)$.
\end{lemma}

\begin{proof}{of Theorem \ref{disisisii}}
	Suppose that $h \in \ho_\zeta(\Sigma)$.	Let $(\{\mu_{0, \gamma}\}_\gamma, \phi_x)$ be the disintegration of $\mu _0$ and denote by $\nu$ the measure $\nu:=h\mu_0$ ($ h\mu_0(E) := \int _E {h}d\mu _0$). If $(\{\nu_{ \gamma}\}_\gamma, \widehat{\nu} )$ is the disintegration of $\nu$, then it holds $\widehat{\nu} \ll m_1$ and $\nu _\gamma \ll \mu_{0, \gamma}$ (see Lemma \ref{hdgfghddsfg}). Moreover, denoting $\overline{h}:=\dfrac{d\widehat{\nu}}{dm_1}$ as (\ref{fjgh}) and
	\begin{equation*}
		\dfrac{d\nu _\gamma}{d\mu_{0, \gamma}} (y) =
		\begin{cases}
			\dfrac{h(\gamma,y) }{\overline{h}(\gamma)} ,& \ \ \textnormal{if} \ \  	\overline{h}(\gamma) \neq 0 \\
			& \\
			\quad  0,&\ \ \textnormal{if} \ \  	\overline{h}(\gamma) = 0.
		\end{cases}
	\end{equation*}
	It is immediate that $\nu \in \mathcal{L}^\infty$. Let us check that $\overline{h} \in H_\zeta$ by estimating the H\"older constant of $\overline{h}$. 
	For $\gamma _1, \gamma _2 \in M$, we have
	\begin{eqnarray*}
		|\overline{h}(\gamma_2) - \overline{h}(\gamma_1)| &\leq& \left|\int_{M}{h(\gamma _2, y)}d(\mu_0|_{\gamma_2}) - \int_{M}{h(\gamma _{1}, y)}d(\mu_0|_{\gamma_{1}}) \right| 
		\\&\leq& \left|\int_{M}{h(\gamma _2, y)}d(\mu_0|_{\gamma_2}) - \int_{M}{h(\gamma _{2}, y)}d(\mu_0|_{\gamma_{1}}) \right| 
		\\&+& \left|\int_{M}{h(\gamma _2, y)}d(\mu_0|_{\gamma_{1}}) - \int_{M}{h(\gamma _{1}, y)}d(\mu_0|_{\gamma_{1}}) \right| 
		\\&\leq& \left|\int_{M}{h(\gamma _2, y)}d(\mu_0|_{\gamma_2}-\mu_0|_{\gamma_{1}}) \right| 
		\\&+& \left|\int_{M}{h(\gamma _2, y)-h(\gamma _{1}, y)}d(\mu_0|_{\gamma_{1}}) \right|
		\\&\leq& ||h||_{\zeta} ||\mu_0|_{\gamma_2}-\mu_0|_{\gamma_{1}}||_W + H_\zeta(h)d_1(\gamma _2, \gamma _{1}) ^\zeta \left|\phi _x \right|_{\infty}
		\\&\leq& ||h||_{\zeta} 	|\mu_0|_\zeta d_1(\gamma _2, \gamma _{1})^\zeta + H_\zeta(h)d_1(\gamma _2, \gamma _{1})^\zeta  \left|\phi _x \right|_{\infty}.
	\end{eqnarray*}Thus, $\overline{h} \in H_\zeta$.  
\end{proof}

\section{Appendix 1: On Disintegration of Measures}\label{disint}

In this section, we prove some results on disintegration of absolutely continuous measures with respect to a measure $\mu_0 \in \mathcal{AB}$. Precisely, we are going to prove Lemma \ref{hdgfghddsfg}. We warn the reader that in this section and the next we work with a slightly different notation.

Let us fix some notations. Denote by $(N_1,m_1)$ and $(N_2,m_2)$ the spaces defined in section \ref{sec2}. For a $\mu_0$-integrable function $f: N_1 \times N_2 \longrightarrow \mathbb{R}$ and a pair $(\gamma,y) \in N_1 \times N_2$ ($\gamma \in N_1$ and $y \in N_2$) we denote by $f_\gamma : N_2 \longrightarrow \mathbb{R}$, the function defined by $f_\gamma(y) = f(\gamma,y)$ and $f|_\gamma$ the restriction of $f$ on the set $\{\gamma\} \times N_2$. Then $f_\gamma = f|_\gamma \circ \pi _{y,\gamma}^{-1}$ and $f_\gamma \circ \pi_{y,\gamma} = f|_\gamma$, where $\pi _{y,\gamma}$ is restriction of the projection $\pi _y(\gamma,y):=y$ on the set $\{\gamma\} \times N_2$. When no confusion is possible, we will denote the leaf $\{\gamma\} \times N_2$, just by $\gamma$.

From now on, for a given positive measure $\mu \in \mathcal{AB}$, on $N_1 \times N_2$, $\widehat{\mu}$ stands for the measure $\pi_x {_*} \mu$, where $\pi_x$ is the projection on the first coordinate, $\pi_x(x,y)=x$.

For each measurable set $A \subset N_1$, define $g: N_1 \longrightarrow \mathbb{R}$ by $$g(\gamma)= \phi_x(\gamma)\int{\chi _{ \pi _x ^{-1}(A)}|_\gamma (y) f|_\gamma(y)}d\mu_{0,\gamma}(y)$$and note that

\begin{equation*}
g(\gamma)=
\begin{cases}
\phi_x(\gamma) \displaystyle{\int {f|_\gamma (y)}d \mu_{0,\gamma}},& \text{if } \gamma \in A \\
0,& \text{if } \gamma \notin A. 
\end{cases}
\end{equation*}Then, it holds $$g(\gamma) = \chi _A (\gamma) \displaystyle{\phi_x(\gamma)  \int {f|_\gamma (y)}d \mu_{0,\gamma}}.$$

\begin{proof}{(of Lemma \ref{hdgfghddsfg})}

	For each measurable set $A \subset N_1$, we have
	\begin{eqnarray*}
		\int_A{\dfrac{\pi _x ^* (f\mu_0)}{dm_1}}dm_1 &=&\int{\chi _A \circ \pi _x}d(f\mu_0)
		\\&=&\int{\chi _{ \pi _x ^{-1}(A)} f}d\mu_0
		\\&=&\int \left[\int{\chi _{ \pi _x ^{-1}(A)}|_\gamma (y) f|_\gamma(y)}d\mu_{0,\gamma}(y)\right]d(\phi_x m_1)(\gamma)
		\\&=&\int \left[\phi_x(\gamma)\int{\chi _{ \pi _x ^{-1}(A)}|_\gamma (y) f|_\gamma(y)}d\mu_{0,\gamma}(y)\right]d(m_1)(\gamma)
		\\&=&\int{g(\gamma)}d(m_1)(\gamma)
		\\&=&\int_A {\left[\int{f_\gamma(y)}d\mu_{0}{|_\gamma}(y)\right]}d(m_1)(\gamma).
	\end{eqnarray*}Thus, it holds 
	
	\begin{equation*}
	\dfrac{\pi _x {_*} (f\mu_0)}{dm_1} (\gamma) =  \int{f_\gamma(y)}d\mu_{0}{|_\gamma}, \ \hbox{for} \ m_1-\hbox{a.e.} \ \gamma \in N_1.
	\end{equation*}And by a straightforward computation 
	
	\begin{equation}\label{gh}
	\dfrac{\pi _x {_*} (f\mu_0)}{dm_1}(\gamma) = \phi_x (\gamma) \int{f|_\gamma(y)}d\mu_{0,\gamma}, \ \hbox{for} \ m_1-\hbox{a.e.} \ \gamma \in N_1.
	\end{equation}Thus, equation (\ref{fjgh}) is established.
	
	\begin{remark}\label{ghj}
		Setting
		\begin{equation}\label{tyu}
		\overline{f}:= \dfrac{\pi _x {_*} (f\mu_0)}{dm_1},
		\end{equation}we get, by equation (\ref{gh}), $\overline{f}(\gamma)=0$ iff $\phi_x (\gamma) = 0$ or $\displaystyle{\int{f|_\gamma(y)}d\mu_{0,\gamma} (y)=0}$, for  $m_1$-a.e. $\gamma \in N_1$. 
	\end{remark}

	Now, let us see that, by the $\widehat{\nu}$-uniqueness of the disintegration, equation (\ref{gdfgdgf}) holds. To do it, define, for $m_1$-a.e. $\gamma \in N_1$, the function $h_\gamma : N_2 \longrightarrow \mathbb{R}$, in a way that

	\begin{equation*}\label{jri}
	h_\gamma (y)=
	\begin{cases}
	\dfrac{f|_\gamma (y)}{\int{f|_\gamma(y)}d\mu_{0,\gamma}(y)},& \text{if } \gamma \in B ^c \\
	0,& \text{if } \gamma \in B. 
	\end{cases}
	\end{equation*}Let us prove equation (\ref{gdfgdgf}) by showing that, for all measurable set $E \subset N_1 \times N_2$, it holds $$f \mu _0 (E)  = \int _{N_1} {\int _{E \cap \gamma} {h_\gamma(y)}}d\mu _{0, \gamma} (y)d (\pi_x {_*}(f \mu_0))(\gamma).$$In fact, by equations (\ref{gh}), (\ref{tyu}), (\ref{jri}) and remark \ref{ghj}, we get
	
	\begin{align*}
		f\mu_0 (E)\hskip -2em &\qquad= \int _E {f} d\mu_0
		\\&= \int _{N_1} \int _{E\cap \gamma} {f|_\gamma} d\mu_{0, \gamma}d (\phi_x m_1)(\gamma)
		\\&= \int _{B^c} \int _{E\cap \gamma} {f|_\gamma} d\mu_{0, \gamma}d (\phi_x m_1)(\gamma)
		\\&= \int _{B^c} \int{f|_\gamma(y)}d\mu_{0,\gamma}(y)\phi_x (\gamma)\left[ \dfrac{1}{\int{f|_\gamma(y)}d\mu_{0,\gamma}(y)}\int _{E\cap \gamma} {f|_\gamma}  d\mu_{0, \gamma}\right]d m_1(\gamma)
		\\&= \int _{B^c} \overline{f}(\gamma)\left[ \dfrac{1}{\int{f|_\gamma(y)}d\mu_{0,\gamma}(y)}\int _{E\cap \gamma} {f|_\gamma}  d\mu_{0, \gamma}\right]d m_1(\gamma)
		\\&= \int _{B^c} \left[ \dfrac{1}{\int{f|_\gamma(y)}d\mu_{0,\gamma}(y)}\int _{E\cap \gamma} {f|_\gamma}  d\mu_{0, \gamma}\right]d \overline{f}m_1(\gamma)
		\\&= \int _{B^c} {\int _{E \cap \gamma} {h_\gamma(y)}}d\mu _{0, \gamma} (y)d (\pi_x {_*}(f \mu_0))(\gamma)
		\\&= \int _{N_1} {\int _{E \cap \gamma} {h_\gamma(y)}}d\mu _{0, \gamma} (y)d (\pi_x {_*}(f \mu_0))(\gamma).
	\end{align*}And we are done.
\end{proof}

\end{document}